\documentclass{amsart}

\usepackage{amsmath,amssymb,latexsym}

\usepackage{amsfonts,amsthm}

\usepackage{mathrsfs}

\usepackage[all]{xy}

\author{Micah J. Leamer}  

\title{Torsion and tensor products over domains and specializations to semigroup rings}
\begin{document}

\theoremstyle{plain}
\newtheorem{prop}{Proposition}[section]
\newtheorem{thm}[prop]{Theorem}
\newtheorem{lem}[prop]{Lemma}
\newtheorem{cor}[prop]{Corollary}

\theoremstyle{definition}
\newtheorem{remark}[prop]{Remark}
\newtheorem{notation}[prop]{Notation}
\newtheorem{defn}[prop]{Definition}
\newtheorem{alg}[prop]{Algorithm}
\newtheorem{question}[prop]{Question}
\newtheorem{ex}[prop]{Example}
\newtheorem{exs}[prop]{Examples}
\newtheorem{conj}[prop]{Conjecture}
\newtheorem{HWconj}[prop]{Conjecture \cite[473--474]{Huneke}}

\renewcommand{\geq}{\geqslant}
\renewcommand{\leq}{\leqslant}

\newcommand{\B}[1]{\mathcal{B}(#1)}
\newcommand{\BH}[1]{\mathcal{B}_X(H(#1))}
\newcommand{\MB}[1]{\operatorname{Max\mathcal{B}}(#1)}
\newcommand{\Pk}[2]{\operatorname{Peak}({#1},{#2})}
\newcommand{\GZ}[1]{\operatorname{Gen}(#1)}
\newcommand{\G}[1]{\operatorname{Gen}(#1)}
\newcommand{\len}[1][R]{\lambda_{#1}}
\newcommand{\Ext}[4][R]{\operatorname{Ext}_{#1}^{#2}(#3,#4)}	
\newcommand{\Hom}[3][R]{\operatorname{Hom}_{#1}(#2,#3)}	
\newcommand{\Tor}[4][R]{\operatorname{Tor}^{#1}_{#2}(#3,#4)}
\newcommand{\LC}{\operatorname{H}}
\newcommand{\T}{\operatorname{T}}
\newcommand{\E}{\operatorname{E}}

\newcommand{\coker}{\operatorname{coker}}
\newcommand{\comp}[1]{\widehat{#1}}
\newcommand{\m}{\mathfrak{m}}
\newcommand{\n}{\mathfrak{n}}
\newcommand{\kspan}{k\operatorname{-span}}
\newcommand{\cmp}{\operatorname{comp}}
\newcommand{\supp}{\operatorname{supp}}
\newcommand{\mspec}{\operatorname{m-spec}}
\newcommand{\im}{\operatorname{Im}}
\newcommand{\depth}{\operatorname{depth}}
\newcommand{\Ap}{\operatorname{Ap}}
\newcommand{\Hilb}{\operatorname{H}}


\begin{abstract} Let $R$ be a commutative Noetherian domain, and let $M$ and $N$ be finitely generated $R$-modules.  We give new criteria for determining when $M\otimes_RN$ has torsion. We also give constructive formulas for producing a module in the isomorphism class of $\T(M\otimes_RN)$, where $\T(-)$ gives the torsion submodule of a module. In some cases we determine bounds on the length and minimal number of generators of $\T(M\otimes_RN)$. We focus on the case where $R$ is a numerical semigroup ring with the goal of making progress on the Huneke-Wiegand Conjecture.
\end{abstract}

\maketitle

 Throughout $R$ will be a commutative domain. Additionally $M$ and $N$ will denote non-zero finitely generated $R$-modules. The torsion submodule of $M$ is the set 
$\T(M)=\{x\in M|\ rx=0 \textrm{ for some non-zero }r\in R\}$. 
We say that $M$ is torsion-free provided that $\T(M)=0$. Otherwise we say that $M$ has torsion. In this paper we will use the convention that local rings are Noetherian.

It is often the case that the tensor product of two modules has torsion. Over some classes of rings it has been shown that the only cases where $M\otimes_RN$ is torsion free are trivial.  In particular, when $R$ is either a regular local ring or a one-dimensional local hypersurface domain, then $M\otimes_RN$ is torsion-free if and only if one of $M$ or $N$ is free and the other is torsion-free; see  \cite{Au}, \cite{Li} and \cite{Huneke}. On the other hand simple examples show that this property does not in general extend to complete intersection domains of codimension greater than one.  For instance when $R=k[t^4,t^5,t^6]$  the module $(t^4R+t^5R)\otimes_R(t^4R+t^6R)$ is torsion-free.

We determine criteria on $M$, $N$ and $R$ that often allow one to predict whether $M\otimes_RN$ has torsion.  In some special cases we give an explicit formula for constructing $\T(M\otimes_RN)$ up to isomorphism, and we determine bounds on the length and minimal number of generators of $\T(M\otimes_RN)$.  In other cases we show how to reduce determining the length of $\T(M\otimes_RN)$ to an equivalent number-theoretic problem involving numerical semigroups. 

When $R$ is local and integrally closed  and $M$ is a torsion-free $R$-module, we have that $M\otimes_RM^*$ is reflexive if and only if $M$ is free \cite[3.3]{Au}, where $M^*:=\Hom{M}{R}$.  
C. Huneke and R. Wiegand have conjectured \cite[473--474]{Huneke} that if $R$ is a local domain such that $M$ and  $M\otimes_RM^*$ are Maximal Cohen-Macaulay (MCM), then $M$ is free. Furthermore O. Celikbas and R. Takahashi have shown that if the Huneke-Wiegand conjecture holds when $R$ is one-dimensional and Gorenstein, then it would hold for all Gorentein domains, and the Auslander-Reiten conjecture would also hold over Gorenstein domains \cite[Proposition 5.6]{Celikbas}.  The Auslander-Reiten Conjecture claims that, if $\Ext{i}{M}{R\oplus M}=0$ for all $i>0$, then $M$ is projective. Over a one-dimensional Gorenstein domain a module is reflexive if and only if it is torsion-free, and also if and only if it is MCM. In particular we would like to show that if $R$ is a one-dimensional local Gorenstein domain, and $M$ is torsion-free but not free, then $M\otimes_RM^*$ has torsion.


Trying to make progress on the Huneke-Wiegand conjecture was the original motivation for this research.
When $R$ is any commutative domain and $I$ is a two generated ideal of $R$, we obtain the isomorphism $\T(I\otimes I^*)\cong (I^2)^{-1}/(I^{-1})^2$, where $I^{-1}:=(R:_KI)$ and $K$ is the field of fractions; see Lemma \ref{lem:iso2gen} for details. 
In particular given a numerical 
semigroup $S$ and a field $k$ the following statements are equivalent; see Proposition \ref{prop:arithmetic} for details.
\begin{enumerate}
\item  The Huneke-Wiegand Conjecture holds for two-generated monomial ideals over $k[S]$.
\item For every $n$ in $\mathbb{N}\setminus S$ there exists a set of the form $\{x,x+n,x+2n\}\subset S$, which does not factor as the sum of two sets of the form $\{y,y+n\}\subset S$ and $\{z,z+n\}\subset S$.
\end{enumerate}
In \cite{Sanchez} P. Garc\'ia-S\'anchez and the author use this equivalence to show that two-generated monomial ideals over complete intersection numerical semigroup rings satisfy the Huneke-Wiegand Conjecture.

Much of the previous work related to torsion and tensor products has been focused on trying to determine when 
$\T(M\otimes N)\neq 0$. However, we develop tools that allow us to produce bounds on the size of $\T(M\otimes_RN)$ in some special cases.  For instance when $R$ is a hypersurface numerical semigroup ring with non-principal monomial ideals $I$ and $J$, we show that $\len(\T(I\otimes_R J))\geq \frac{1}{2}\mu(I)\mu(J)$. Here $\len(-)$ gives the length of a module.  Lastly, when $R$ is a hypersurface numerical semigroup ring with monomial ideal $J$, we show that $J\otimes_RJ^*$ has a minimal generating set such that $2\mu(J)-2$ of the generators are torsion elements.  Here $\mu(-)$ gives the minimal number of generators of a module.

\section{Torsion over domains}
Let $I$ be an ideal of $R$, let $M$ be a finitely generated torsion-free $R$-module and let $\pi_{IM}:I\otimes_{R}M\to IM$  be the $R$-module homomorphism defined by setting $\pi_{IM}(r\otimes x)=rx$ for all $r$ in $I$ and $x$ in $M$.  When $I$ and $M$ are unambiguous we will simply write $\pi$ for $\pi_{IM}$. In general $\ker(\pi_{IM})=\T(I\otimes_R M)$. Since $IM$ is torsion-free it follows that  $\T(I\otimes_R M)\subseteq\ker(\pi_{IM})$. Conversely, given $x$ in $I$ and $f=\sum_i x_i\otimes y_i$ in $\ker(\pi_{IM})$ we have $xf=\sum_i xx_i\otimes y_i=x\otimes\sum_i x_iy_i=0$.

A local ring is said to be analytically irreducible if its completion is a domain.  An analytically irreducible local ring is said to be residually rational if it has the same residue field as its integral closure. Note that in this case the integral closure would necessarily be local by \cite[6]{Katz}.

%


\begin{remark} \label{lem:KrullKatz}
Let $R$ be a one-dimensional Noetherian local domain. Then $R$ is analytically irreducible if and only if its integral closure $\overline{R}$ is finitely generated as an $R$-module and is a discrete valuation ring (DVR).

If $R$ is analytically irreducible then $\comp{R}$ is reduced; hence by \cite{Krull}, $\overline{R}$ is a finite $R$-module.
By \cite[6]{Katz} the number of minimal primes in 
$\comp{R}$ equals the number of maximal ideals in $\overline{R}$. Thus if $R$ is analytically irreducible, then $\overline{R}$ is a one-dimensional local integrally closed domain; hence by \cite[Proposition 9.2]{Atiyah}, $\overline{R}$ is a DVR. Conversely if $R$ is not analytically irreducible, then $\overline{R}$ is not local and therefore not a DVR.  
\end{remark}

\begin{lem} \label{prop:moduleOneWay}
Let $R$ be a domain. Let $M$ be a finitely generated torsion-free $R$-module.   Let $I,P$ and $Q$ be finitely generated  ideals such that $P+Q=I$. Then the sequence
\[
\xymatrix{
\T(P\otimes_{R} M)\oplus\T(Q\otimes_{R} M)
\ar[rr]^{\ \ \ \ \ \ \ [a,b]\mapsto a+b} & & 
\T(I\otimes_{R} M) \ar[r]^{\delta} &
\frac{(PM)\cap(QM)}{(P\cap Q) M}\ar[r] 
& 0}.
\] 
is exact.
 In particular, if $I=(f,g)$, then $\displaystyle\T(I\otimes_RM)\cong\frac{(fM)\cap(gM)}{(fR\cap gR)M}$.
\end{lem}

\begin{proof} 
We have the following exact sequence:
\[
\xymatrix{0\ar[r] & P\cap Q\ar[rr]^{f\mapsto [f,-f]} & &
P\oplus Q\ar[rr]^{[p,q]\mapsto p+q} & & I \ar[r] & 0}
\]
Applying $(-)\otimes_{R}M$ we get the second row of the following commutative exact diagram:
\[
\xymatrix@=1.8em{ 
& &\T(P\otimes_{R} M)\oplus \T(Q\otimes_{R} M)\ar @{^{(}->}[d] &\T(I\otimes_{R} M)\ar @{^{(}->}[d]\\
& (P\cap Q)\otimes_{R} M\ar[d]\ar[r] & (P\otimes_{R} M)\oplus (Q\otimes_{R}M)\ar[d]^{\pi_{P,M}\oplus \pi_{Q,M}} \ar[r] & I\otimes_{R} M \ar[d]^{\pi_{IM}} \ar[r] & 0\\
0 \ar[r] & (P M)\cap(Q M) \ar[r]^{f\mapsto [f,-f]} \ar @{->>}[d] & P M\oplus Q M \ar[d] \ar[r]^{[p,q]\mapsto p+q} & IM \ar[r] \ar[d] & 0\\
& \frac{(PM)\cap(QM)}{(P\cap Q)M} & 0 & 0}
\]
Thus by the Snake Lemma we get the desired exact sequence. Note that the map $\delta$ in the lemma is the connecting map from the Snake Lemma. 
\end{proof}

Let $R$ be a local domain with maximal ideal $\m$. Suppose there exists a fixed $t$ in $\overline{R}$ such that $\m=(t^{n_1},\hdots ,t^{n_e})$. Then $I$ is said to be a fractional \emph{monomial} ideal whenever $I=(t^{z_1},\hdots,t^{z_h})$ for some integers $z_1,\hdots ,z_h$.  

\begin{thm} \label{prop:moduleOneWay2}
Suppose that $R$ is either a $\mathbb{Z}^n$ standard graded $k$-subalgebra of $k[x_1,\hdots,x_n]$ or a one-dimensional analytically irreducible residually rational domain with maximal ideal $\m=(t^{n_1},\hdots ,t^{n_e})$, for some $t$ in $\overline{R}$. Let $I=(a_1,\hdots,a_m)$ and $J$ be finitely generated monomial fractional ideals, where each $a_i$ is a monomial.  Let $\mathbf{e}_1,\hdots ,\mathbf{e}_m$ be a basis for $R^m$.  Then 
$\displaystyle \T(I\otimes_RJ)\cong\frac{\sum_{i<j}(a_iJ\cap a_jJ)(\mathbf{e}_i-\mathbf{e}_j)}{\sum_{i<j}(a_iR\cap a_jR)J(\mathbf{e}_i-\mathbf{e}_j)}$.
\end{thm}

\begin{proof} 
We claim that the sequence
\[
\xymatrix{0\ar[r] & \sum_{i<j}(a_iR\cap a_jR)(\mathbf{e}_i-\mathbf{e}_j) \ar[r]^{\ \ \ \ \ \ \ \ \ \subseteq} &
\sum_{i=1}^m a_iRe_i \ar[r]^{\ \ \ \ \  \gamma} &  I \ar[r] & 0}
\]
is exact, where $\gamma$ is the map that replaces direct sums with addition.
It suffices to show that 
$\ker(\gamma)\subseteq \sum_{i<j}(a_iR\cap a_jR) (\mathbf{e}_i-\mathbf{e}_j)$.

Case 1: Let $R$ be a $\mathbb{Z}^n$ standard graded $k$-subalgebra of $k[x_1,\hdots,x_n]$.  Let $f$ be a non-zero homogeneous element of $\ker(\gamma)$ of degree $w=[w_1,\hdots,w_n]$ in $\mathbb{Z}^n$.  Let $x^w$ denote $\prod_{i=1}^nx_i^{w_i}$. Since $f$ is homogeneous and in $\sum_{i=1}^m a_iRe_i$ there exist $\alpha_i$ in $k$ such that $f=x^w\sum_{i=1}^m \alpha_i\mathbf{e}_i$ where $x^w$ in $a_iR$ whenever $\alpha_i\neq 0$. Also since $f$ is in $\ker(\gamma)$ it follows that $\sum_{i=1}^m\alpha_i=0$. Choose $h$ such that $\alpha_h\neq 0$.  Let $S=\{i|\ \alpha_i\neq 0,\ i\neq h\}$. Then
$f=\sum_{i\in S} \alpha_ix^v(\mathbf{e}_i-\mathbf{e}_h)$ is in $\sum_{i<j}(a_iR\cap a_jR)(\mathbf{e}_i-\mathbf{e}_j)$.

Case 2: Let $R$ be a one-dimensional analytically irreducible residually rational domain with maximal ideal $(t^{n_1},\hdots ,t^{n_e})$.  Let $f$ be a non-zero element of $\ker(\gamma)$. Then $f=\sum_{i=1}^mf_i\mathbf{e}_i$ for some $f_i$ in $a_iR$ with $\sum_{i=1}^mf_i=0$.  Let $v$ be the valuation associated to the valuation ring $\overline{R}$ and let $\n$ be the maximal ideal for $\overline{R}$. There exists $N\gg 0$ such that $\n^N\subset a_iR$ for all $i$. When $v(f_i)>N$ for all $i$, it follows that  
$f$ is in $\sum_{i<j}(a_iR\cap a_jR)(\mathbf{e}_i-\mathbf{e}_j)$.  Let $d=\min\{v(f_i)|\ i=1,\hdots ,m\}$. By induction it suffices to show that there exists $f'=\sum f'_i\mathbf{e}_i$, such that $v(f'_i)>d$ for all $i$ and such that we can write $f'$ as $f-g$ for some $g$ in $\sum_{i<j}(a_iR\cap a_jR) (\mathbf{e}_i-\mathbf{e}_j)$. Let $S=\{i|\ v(f_i)=d\}$. Choose $h$ in $S$. For all $j$ in $S\setminus\{h\}$ there exists a unit $u_j$ in $R$ such that 
$v(f_j-u_jt^d)>d$. We claim that $f'=f-\sum_{j\in S\setminus\{h\}}u_jt^d(\mathbf{e}_h-\mathbf{e}_j)$ has the desired properties. By construction $v(f'_j)>d$ for all $j\neq h$.  As $\sum_{i=1}^mf'_i=0$, we have 
$v(f'_h)\geq\min\{v(f'_j)|\ j\neq h\}>d$, and the claim follows.     

Applying $(-)\otimes_{R}J$ to the exact sequence above we get the second row of the following commutative exact diagram:
\[
\xymatrix@=1.6em{ 
& & 0 \ar[d] &\T(I\otimes_{R} J)\ar @{^{(}->}[d]\\
& \sum_{i<j}(a_iR\cap a_jR)\otimes_{R} J (\mathbf{e}_i-\mathbf{e}_j)\ar[d]\ar[r] & \bigoplus_{i=1}^m a_iR\otimes_{R} J \ar[d] \ar[r] & I\otimes_{R} J \ar[d] \ar[r] & 0\\
0 \ar[r] & \sum_{i=1}^m(a_iJ\cap a_jJ)(\mathbf{e}_i-\mathbf{e}_j)\ar[r] \ar[d] & \bigoplus_{i=1}^ma_iJ \ar @{->>}[d] \ar[r] & IJ \ar[r] \ar[d] & 0\\
& \frac{\sum_{i<j}(a_iJ\cap a_jJ)(\mathbf{e}_i-\mathbf{e}_j)}{\sum_{i<j}(a_iR\cap a_jR)J(\mathbf{e}_i-\mathbf{e}_j)} & 0 & 0}
\]
The desired isomorphism follows from the Snake Lemma.
\end{proof}

\begin{thm} \label{thm:equiv}
Suppose either that $R$ is a $\mathbb{Z}^n$ standard graded $k$-subalgebra of $k[x_1,\hdots,x_n]$ or that $(R,\m)$ is an analytically irreducible residually rational ring with maximal ideal $\m=(t^{n_1},\hdots ,t^{n_e})$ for some $t$ in $\overline{R}$. Let $I$ and $J$ be finitely generated monomial fractional ideals of $R$ and let $I=(a_1,\hdots ,a_h)$ such that each $a_i$ is a monomial. Then the following conditions are equivalent:
\begin{enumerate}
\item \label{equiv1} $\T(I\otimes_RJ)=0;$
\item \label{equiv2} $(P\cap Q)J=PJ\cap QJ$ for all ideals $P$ and $Q$ (not necessarily monomial) such that $P+Q=I$; and 
\item \label{equiv3} $((a_i|\ i\in \mathcal{S})\cap(a_j|\ j\notin \mathcal{S}))J=(a_i|\ i\in \mathcal{S})J\cap(a_j|\ j\notin \mathcal{S})J$\\ for all $\mathcal{S}\subset\{1,\hdots ,h\}$.
\end{enumerate}
\end{thm}

\begin{proof}
$(1)\implies(2)$: This follows from Lemma \ref{prop:moduleOneWay}.

$(2)\implies(3)$: This is clear, since (3) is a special case of (2).

${(3)\implies(1)}$: 
Let $R$ be a $\mathbb{Z}^n$ standard graded $k$-subalgebra of $k[x_1,\hdots,x_n]$. 
Fix $w$ in $\mathbb{Z}^n$ and let $V=\{i|\ x^w\in a_iJ\}$. Let $G$ be the graph with vertex set $V$ that contains the edge $ij$ if and only if $x^w$ is in $(a_iR\cap a_jR)J$. Let $\{\mathcal{S},\mathcal{S}'\}$ be a partition of $\{1,\hdots ,h\}$.  Assume that $((a_i|\ i\in \mathcal{S})\cap(a_j|\ j\in \mathcal{S}'))J=(a_i|\ i\in \mathcal{S})J\cap(a_j|\ j\in \mathcal{S}')J$. If 
$(a_i|\ i\in \mathcal{S})J\cap(a_j|\ j\in \mathcal{S}')J$ does not contain $x^w$, then either $\mathcal{S}\cap V=\emptyset$ or $\mathcal{S}'\cap V=\emptyset$. If $((a_i|\ i\in \mathcal{S})\cap(a_j|\ j\in \mathcal{S}'))J$ contains $x^w$, then $G$ contains an edge between a vertex in $\mathcal{S}$ and a vertex in $\mathcal{S}'$.  Since this occurs for all $\mathcal{S}\subset\{1,\hdots,h\}$, it follows that $G$ is path connected.

Suppose that $x^w(\mathbf{e}_y-\mathbf{e}_z)$ is in $\sum_{i<j}(a_iJ\cap a_jJ)(\mathbf{e}_i-\mathbf{e}_j)$ for some $y,z$ in $\{1,\hdots,h\}$. It follows that $y$ and $z$ are in $V$.  Since $G$ is path connected, there is a path $i_1i_2\hdots i_h$ from $y$ to $z$ where $i_j$ is in $V$, $i_1=y$ and $i_h=z$. By our definition of $G$, for each edge $i_{j}i_{j+1}$, we have that $x^w(\mathbf{e}_{i_j}-\mathbf{e}_{i_{j+1}})$ is in $\sum_{i<j}(a_iR\cap a_jR)J(\mathbf{e}_i-\mathbf{e}_j)$. Thus 
\[
\textstyle x^w(\mathbf{e}_y-\mathbf{e}_z)=\sum_{j=1}^{h-1}x^w(\mathbf{e}_{i_j}-\mathbf{e}_{i_{j+1}})\in\sum_{i<j}(a_iR\cap a_jR)J(\mathbf{e}_i-\mathbf{e}_j).
\] 
Since this condition holds for all $w$ in $\mathbb{Z}^n$ it follows that 
\[
\textstyle\sum_{i<j}(a_iJ\cap a_jJ)(\mathbf{e}_i-\mathbf{e}_j)=\sum_{i<j}(a_iR\cap a_jR)J(\mathbf{e}_i-\mathbf{e}_j).
\]
Thus by Theorem \ref{prop:moduleOneWay2} we have $\T(I\otimes_RJ)=0$. 

The case where $R$ is a one-dimensional analytically irreducible residually rational ring  with maximal ideal $\m=(t^{n_1},\hdots ,t^{n_e})$ is analogous.  Pick $w$ in $\mathbb{Z}$ and use $t^w$ instead of $x^w$ in the argument above. 
\end{proof}

\begin{question}
From Lemma \ref{prop:moduleOneWay} it follows that the implication $\eqref{equiv1}\implies\eqref{equiv2}$ in Theorem \ref{thm:equiv} remains true for general ideals over any commutative domain.  For which ideals and which classes of rings is the reverse implication also true? 
\end{question}
%
%
%


\begin{defn}
Let $R$ be a domain with field of fractions $K$.  A \emph{fractional ideal} $I$ is a finitely generated submodule of $K$.  Let $M$ be a finitely generated rank $n$ submodule of $K^n$. The \emph{inverse} of $M$ is $M^{-1}:=\{v\in K^n|\ v\cdot w\in R$ for all $w\in M\}$, where $\cdot$ is the dot product. In particular $I^{-1}:=(R:_KI)$.
\end{defn}

\begin{remark}
Let $R$ be a domain with field of fractions $K$. Let $M$ and $N$ be rank $n$ submodules of $K^n$. Then $(M+N)^{-1}=M^{-1}\cap N^{-1}$.

Let $x$ be an element of $(M+N)^{-1}$. Then $x\cdot M\subseteq R$ and $x\cdot N\subseteq R$, so $x$ is in $M^{-1}\cap N^{-1}$. Let $y$ be an element of $M^{-1}\cap N^{-1}$. Then $y\cdot (M+N)=y\cdot M+y\cdot N\subseteq R$, so $y$ is in $(M+N)^{-1}$.
\end{remark}

\begin{lem}
Let $R$ be a domain with field of fractions $K$. Let $M$ be a finitely generated submodule of $K^n$. Then there is a natural isomorphism $M^{-1}\cong\Hom{M}{R}$ defined by sending $v$ in $M^{-1}$ to the map $[w\mapsto v\cdot w]$.
\end{lem}

\begin{proof}
The lemma will follow if we can show that the map $[w\mapsto v\cdot w]$ is invertible. Let $e_1,\hdots e_n$ be standard basis vectors in $K^n$. Since $M$ has rank $n$ there exist 
$\alpha_1,\hdots,\alpha_n$ in $K\setminus\{0\}$ such that $\alpha_1e_1,\hdots ,\alpha_ne_n$ are in $M$. By clearing denominators we may choose the $\alpha_i$ to be in $R$. Our candidate for the inverse map will map $f$ in $\Hom{M}{R}$ to the vector
$\sum_{i=1}^n \frac{f(\alpha_ie_i)e_i}{\alpha_i}$ in $K^n$.  

Any element of $x$ in $M$ can be written in the form $x=\sum_{i=1}^n\beta_i\alpha_ie_i$ for some $\beta_i$ in $K$. Since $f$ is $R$-linear it must also be $K$-linear. This explains the last step in the next display.
\begin{align*}
\textstyle x\cdot \sum_{i=1}^n \frac{f(\alpha_ie_i)e_i}{\alpha_i}
&\textstyle =\left(\sum_{i=1}^n\beta_i\alpha_ie_i\right)\cdot\left(\sum_{i=1}^n\frac{f(\alpha_ie_i)e_i}{\alpha_i}\right)\\
&\textstyle =\sum\beta_if(\alpha_ie_i)\\
&\textstyle =f\left(\sum_{i=1}^n\beta_i\alpha_ie_i\right)\\
&=f(x)
\end{align*}
Therefore, $[w\mapsto v\cdot w]$ maps the vector $\sum_{i=1}^n \frac{f(\alpha_ie_i)e_i}{\alpha_i}$ back to $f$.
Since $f$ is in $\Hom{M}{R}$ it follows that $f(\sum_{i=1}^n\beta_i\alpha_ie_i)$ is in $R$. Thus $x\cdot \sum_{i=1}^n \frac{f(\alpha_ie_i)e_i}{\alpha_i}$ is in $R$ and $\sum_{i=1}^n \frac{f(\alpha_ie_i)e_i}{\alpha_i}$ is in $M^{-1}$.  Let $v$ be a vector in $M^{-1}$. Then $v=\sum_{i=1}^n\frac{(v\cdot \alpha_ie_i)e_i}{\alpha_i}$, proving that composition in the other direction is also the identity.
\end{proof}


\begin{lem}\label{lem:iso2gen}
Let $R$ be a domain with field of fractions $K$. Let $I=(f,g)$ be a two-generated fractional ideal of $R$ and let $M$ be a rank $n$ submodule of $K^n$.  Then
$\displaystyle\T(I\otimes_RM^*)\cong\frac{(IM)^{-1}}{I^{-1}M^{-1}}$.
Specifically $\displaystyle\T(I\otimes_RI^*)\cong \frac{(I^2)^{-1}}{(I^{-1})^2}$.
\end{lem}

\begin{proof}
By Lemma \ref{prop:moduleOneWay} and the isomorphism $M^*\cong M^{-1}$ we get the first step below.
\begin{align}
\T(I\otimes_RM^*)\cong\frac{fM^{-1}\cap gM^{-1}}{(fR\cap gR)M^{-1}}\cong\frac{fM^{-1}\cap gM^{-1}}{fgI^{-1}M^{-1}}\cong\frac{fg(IM)^{-1}}{fgI^{-1}M^{-1}}\cong\frac{(IM)^{-1}}{I^{-1}M^{-1}}
\end{align} The equivalence
$fR\cap gR=fg(g^{-1}R\cap f^{-1}R)=fg(gR+fR)^{-1}=fgI^{-1}$ justifies the second step in $(1)$. The
equivalence $fM^{-1}\cap gM^{-1}=fg(g^{-1}M^{-1}\cap f^{-1}M^{-1})=fg(gM+fM)^{-1}=fg(IM)^{-1}$ gives the third step in $(1)$, and the result follows.
\end{proof}

\section{Correspondence between rings and numerical semigroups}

Let $\mathbb{N}_0$ denote the non-negative integers. A numerical semigroup $S$ is a submonoid of 
$(\mathbb{N}_0,+)$ with finite complement in $\mathbb{N}_0$. We use the notation $\langle n_1,\hdots ,n_e\rangle$ for the numerical semigroup $n_1\mathbb{N}_0+\hdots +n_e\mathbb{N}_0$. The Frobenius number of a numerical semigroup $S$ is the largest integer not in $S$ and is denoted by $F_S$ or simply $F$, when the underlying semigroup is unambiguous. 
For a detailed introduction to numerical semigroups see \cite{RosalesSanchez}.

Let $S$ be a submonoid of $\mathbb{N}_0^n$. Then a relative ideal $A$ of $S$ is a set of the form 
$(a_1,\hdots , a_m):=\bigcup_{i=1}^h(a_i+S)$ for some $a_1,\hdots,a_h$ in $\mathbb{Z}^n$.

Let $R$ be a one-dimensional analytically irreducible residually rational  domain with field of fractions $K$. Let 
$(\overline R,\mathfrak{n})$ denote the integral closure of $R$ with maximal ideal $\mathfrak{n}$. In this case $\overline R$ is a DVR. Fix a generator $t$ for $\mathfrak{n}$ and let $v:K^{\times}\to\mathbb{Z}$ be the valuation given by $v(f):=\sup\{i\in\mathbb{Z}|\ f\in t^iR\}$. Let $I$ be a fractional ideal of $R$. We define $v(R):=v(R-\{0\})$ and $v(I):=v(I-\{0\})$.  Then $v(R)$ is a numerical semigroup, and $v(I)$ is a relative ideal of $v(R)$.

%

\begin{defn}
Choose a monomial ordering on $R=k[x_1,\hdots ,x_n]$. Let $f$ be a non-zero element of $R$.  Then $\deg(f):=d=[d_1,\hdots,d_n]$ in $\mathbb{N}^n_0$ where $x^d:=\prod x_i^{d_i}$ is the leading monomial of $f$. For any non-zero elements $f$ and $g$ in $R$ we define $\deg(f/g):=\deg(f)-\deg(g)$ in $\mathbb{Z}^n$. Let $X$ be a subset of the quotient field $K$. Then 
\[\deg(X):=\{d\in\mathbb{Z}^n|\ \deg(f)=d \textrm{ for some } f\in X\setminus\{0\}\}.\]
\end{defn}

Note that we will often be considering $R$ as a graded $k$-subalgebra of $k[x_1,\hdots,x_n]$.  In such cases there may be many monomial orderings on the monomials of $R$, which are not simply restriction of monomial orderings from $k[x_1,\hdots,x_n]$. However, we will only be considering monomial orderings from $k[x_1,\hdots,x_n]$.

Let $R$ be a $k$-subalgebra of $k[x_1,\hdots,x_n]$ and let $I$ be a fractional ideal of $R$. Then $\deg(R)$ is a submonoid of $\mathbb{N}^n_0$ and $\deg(I)$ is a relative ideal of $\deg(R)$.


 \begin{remark} 
 \label{prop:snug} Let $R$ be a one-dimensional analytically irreducible residually rational ring with maximal ideal $\m$. Then we have the following:
\begin{enumerate}
\item \label{snug} The conductor $(R:_K\overline R)$ equals $\mathfrak{n}^{F+1}$, where $F=F_{v(R)}$;
\item \label{unit} If $f$ and $g$ are non-zero elements of $K$ such that 
$v(f)=v(g)$, then there exists a unit $u$ in $R$ such that $v(f-ug)>v(f)$.
\item \label{max} If $\m=(t^{n_1},\hdots ,t^{n_e})$ for some $t$ in $\overline R$,
then $v(R)=\langle n_1,\hdots ,n_e\rangle$.
\end{enumerate}

\eqref{snug}: This is shown in the proof of  \cite[Theorem]{Kunz}. 

\eqref{unit}: Consider the natural map $\gamma:\n^(f)\to\n^{v(f)}/\n^{v(f)+1}$. There exists an element $c$ in $k\setminus\{0\}$ such that $v(f)=cv(g)$. Since  $k=R/\m$ we may choose a unit $u$ in $R$ that maps to $c$ under the natural map $R\to k$.  It follows that $v(f-ug)=v(f)-cv(g)=0$. Thus $f-ug$ is in $\n^{v(f)+1}$ and the result follows.

\eqref{max}: The inclusion $\langle n_1,\hdots ,n_e\rangle\subseteq v(R)$ is straightforward. 
Let $r$ be a non-zero element of $R$. Since $R$ is local, $r=\sum_{i=1}^hu_it^{n\cdot w_i}$ where each $u_i$ is a unit in $R$, $n=[n_1,\hdots ,n_e]$, 
each $w_i$ is in $\mathbb{N}_0^{e}$ and $ n\cdot w_i\leq n\cdot w_{i+1}$. Let $\ell$ be maximal such that $ n\cdot w_\ell=n\cdot w_1$.  Since $ n\cdot w_1\leq v(r)$, we may assume that $ n\cdot w_1$ is maximal among all possible choices for $w_1$. In this case $\textstyle v(\sum_{i=1}^{\ell}u_it^{ n\cdot w_i})
=v((\sum_{i=1}^{\ell}u_i)t^{ n\cdot w_1})= n\cdot w_1$; 
hence $v(r)=n\cdot w_1$ is in
$\langle n_1,\hdots ,n_e\rangle$.
\end{remark}
 

\begin{remark}\label{lem:quotient}\cite[Proposition II.1.4]{Barucci}
Let $R$ be a one-dimensional analytically irreducible residually rational ring. If $I$ and $J$ are fractional ideals of $R$ with $I\subseteq J$, then 
$\len(J/I)=|v(J)\setminus v(I)|$, where $\len$ denotes length as an $R$-module. 
\end{remark}


\begin{remark} \label{rem:quotient}
Let $R$ be a $k$-subalgebra of $k[x_1,\hdots,x_n]$. 
Given fractional ideals $I\subseteq J$ of $R$ we have $\len(J/I)=|\deg(J)\setminus\deg(I)|$. 

   
Since $R$ is a $k$-algebra, $J/I$ is a $k$-vector space.  By choosing a monomial order, modulo $I$ each element of $J$ is equivalent to a unique reduced polynomial such that none of its terms are divisible by a leading monomial in $I$. These polynomials form a $|\deg(J)\setminus\deg(I)|$-dimensional $k$-vector space, which is isomorphic to $J/I$. 
\end{remark}

\begin{defn} 
Let $S$ be a submonoid of $\mathbb{N}_0^n$, and let $A$ and $B$ be relative ideals of $S$.  Then the semigroup tensor product $A\otimes_SB$ is the set $A\times B$ modulo the equivalence relation generated by $(s+a,b)\sim (a,s+b)$ for all $a$ in $A$, $b$ in $B$ and $s$ in $S$. Elements of $A\otimes_S B$ will be written in the form 
$a\otimes b$ with $a$ in $A$ and $b$ in $B$. 

We fix a map $\chi:A\otimes_SB\to\mathbb{Z}^n$ defined by 
$\chi(a\otimes b)=a+b$. 
For each $z$ in $\mathbb{Z}^n$ let
$\tau_z(A,B):=\textrm{max}\{0,  |\chi^{-1}(z)|-1\}$.
Finally, let $\tau(A,B)=\sum_{z\in\mathbb{Z}^n} \tau_z(A,B)$
be the {\it torsion number} of $A$ and $B$.
\end{defn}

\begin{remark}
Let $I$ and $J$ be monomial ideals of a numerical semigroup ring $R$. Then $\tau_z(\deg(I),\deg(J))=\Hilb(\T(I\otimes_RJ),z)$, where $\Hilb(-,z)$ is the Hilbert function.  

 To see this consider the exact sequence 
$0\to \T(I\otimes_RJ)\to I\otimes_RJ\to IJ\to 0$.
From the additivity of the Hilbert function, we have 
\[
\Hilb(\T(I\otimes_RJ),z)=\Hilb(I\otimes_RJ,z)-\Hilb(IJ,z).
\]
Now 
$\Hilb(I\otimes_RJ,z)$ is the same as the number of elements of $\deg(I)\otimes_{\deg(R)}\deg(J)$ which map to $z$ under the natural map $\chi:\deg(I)\otimes_{\deg(R)}\deg(J)\to\mathbb{Z}$. In other words 
$\Hilb(I\otimes_RJ,z)=|\chi^{-1}(z)|$. Also $\Hilb(IJ,z)$ only has two possible values $0$ and $1$, and it is $0$ if and only if $\Hilb(I\otimes_RJ,z)=0$. Therefore $\Hilb(IJ,z)=0$ if $\Hilb(I\otimes_RJ,z)=0$ and $\Hilb(IJ,z)=1$ otherwise. Thus 
\[
\Hilb(I\otimes_RJ,z)-H(IJ,z)=\textrm{max}\{0,  |\chi^{-1}(z)|-1\}:=\tau_z(\deg(I),\deg(J)).
\]
\end{remark}

\begin{defn}
Let $S$ be a numerical semigroup with relative ideals $A$ and $B$ minimally generated by $a_1,a_2,\hdots ,a_m$ and $b_1,b_2,\hdots ,b_n$ respectively. Given an integer $z$, we define a bipartite graph $\Gamma_z(A,B)$  with respective vertex and edge sets
\[
V_z(A,B)=\{v_i|\ z-a_i\in B\}\cup\{w_j|\ z-b_j\in A\}
\textrm{ and } E_z(A,B)=\{v_iw_j\ |\ z-a_i-b_j\in S\}.
\]
\end{defn}

\begin{ex}
Let $S=\langle 5,11\rangle$, $A=(20,21,22)$ and $B=(0,23,24)$. Let  $a_1=20$, $a_2=21$, $a_3=22$, $b_1=0$, $b_2=23$ and $b_4=24$.  Then 
\[
\Gamma_{44}(A,B)=\begin{array}{c} \xymatrix{\bullet_{v_1}\ar @{-}[r] & \bullet_{w_3}\\
\bullet_{v_2} \ar @{-}[r] & \bullet_{w_2}} \end{array}
\]
The easiest way to see this is that $44-a_1-b_3=0$ is in $S$ and $44-a_2-b_2=0$ is in $S$.  However, none of the other edges appear because there are no other cases where $44-a_i-b_j$ is in $S$. Similarly we have
\[
\Gamma_{45}(A,B)=\begin{array}{c} \xymatrix{\bullet_{v_1}\ar @{-}[r] & \bullet_{w_1}\\
\bullet_{v_2} \ar @{-}[r] & \bullet_{w_3}\\
\bullet_{v_3} \ar @{-}[r] & \bullet_{w_2}} \end{array}\quad
\Gamma_{55}(A,B)=\begin{array}{c} \xymatrix{\bullet_{v_1}\ar @{-}[r] \ar @{-}[ddr]& \bullet_{w_1}\\
\bullet_{v_2} \ar @{-}[r] \ar @{-}[dr] & \bullet_{w_2}\\
\bullet_{v_3} \ar @{-}[uur] \ar @{-}[ur] & \bullet_{w_3}} \end{array}
\]

\end{ex}

\begin{prop} \label{prop:graph}
Let $S$ be a numerical semigroup with relative ideals $A$ and $B$.  There is a one to one correspondence between connected components of $\Gamma_z(A,B)$ and elements $a\otimes b$ in $A\otimes_S B$ such that $a+b=z$. The correspondence is given by identifying the connected component containing $v_i$ with $a_i\otimes_S(z-a_i)$ in $A\otimes_SB$.   
\end{prop}

\begin{proof}
It follows from the definition of $\Gamma_z(A,B)$ that $v_iw_j$ is in $E_z(A,B)$ if and only if there exists $a_i\otimes (b_j+g)$ in $A\otimes_SB$ with $g$ in $S$.

 Suppose that $a_i\otimes(b_j+g)$ and $a_s\otimes(b_t+g')$ are elements of $A\otimes_S B$ with 
$a_i+b_j+g=a_s+b_t+g'=z$.  Then it suffices to show that $a_i\otimes(b_j+g)=a_s\otimes(b_t+g')$ if and only if $v_iw_j$ and $v_sw_t$ are in the same connected component of $\Gamma_z(A,B)$.

Suppose that $v_iw_j$ and $v_sw_t$ are in the same connected component of 
$\Gamma_z(A,B)$. Then there exists a finite sequence of edges $\{v_{i_h}w_{j_h}\}_{h=0}^{\ell}$ such that $v_iw_j=v_{i_0}w_{j_0}$, $v_sw_t=v_{i_{\ell}}w_{j_{\ell}}$ and $v_{i_h}w_{j_h}$ is adjacent to $v_{i_{h+1}}w_{j_{h+1}}$ for all $h$. Since 
$v_{i_h}w_{j_h}$ and $v_{i_{h+1}}w_{j_{h+1}}$ are adjacent either $v_{i_h}=v_{i_{h+1}}$ or $w_{j_h}=w_{j_{h+1}}$.  Suppose that for a given value of $h$ we have $v_{i_h}=v_{i_{h+1}}$. Then $a_{i_h}=a_{i_{h+1}}$, so
\[
a_{i_h}\otimes (b_{j_h}+g_h)=a_{i_h}\otimes (z-a_{i_h})
=a_{i_{h+1}}\otimes(z-a_{i_{h+1}})
=a_{i_{h+1}}\otimes(b_{j_{h+1}}+g_{h+1}).
\] 
Suppose that for a given value of $h$ we have $w_{j_h}=w_{j_{h+1}}$. Then $b_{j_h}=b_{j_{h+1}}$, so
\[
(a_{i_h}+g_h)\otimes b_{j_h}=(z-b_{j_h})\otimes b_{j_h}
=(z-b_{j_{h+1}})\otimes b_{j_{h+1}}
=(a_{i_{h+1}}+g_{h+1})\otimes b_{j_{h+1}}.
\]
It follows by induction on $h$ that $a_i\otimes(b_j+g)=a_s\otimes(b_t+g')$.

Now suppose that $a_i\otimes(b_j+g')=a_s\otimes(b_t+g')$.  Then there exists a sequence of ways to write that element 
\[
 a_{i_0}\otimes(b_{j_0}+g_0),\ a_{i_0}\otimes(b_{j_1}+g'_0),\ (a_{i_0}+g'_0)\otimes b_{j_1},\ 
 (a_{i_1}+g_1)\otimes b_{j_1},\hdots ,\ a_{i_{\ell}}\otimes(b_{j_{\ell}}+g_{\ell})
 \]
 such that $a_i=a_{i_0},\ b_j=b_{j_0},\ a_s=a_{i_{\ell}}$ and $b_t=b_{j_{\ell}}$. It follows that the edges 
 $v_iw_j$ and $v_sb_t$ are connected by the adjacent edges 
 $v_{i_0}w_{j_0},\ v_{i_0}w_{j_1},\ v_{i_1}w_{j_1},\ 
 v_{i_1}w_{j_2},\hdots ,\ v_{i_{\ell}}w_{j_{\ell}}$ and the result follows. 
 \end{proof}

\begin{cor}
Let $S$ be a numerical semigroup with relative ideals $A$ and $B$ and let $z$ be an integer. Then $\tau_z(A,B)$ is equal to maximum of $0$ and one less than the number of connected components of $\Gamma_z(A,B)$.
\end{cor}

For the next proposition note that if $R$ is a $k$-subalgebra of $k[x]$, then the condition that 
$\deg(IJ)\setminus (\deg(I)+\deg(J))$ is a always a finite set.

\begin{prop}  \label{prop:MtoG}
Let $R$ be a $k$-subalgebra of $k[x_1,\hdots,x_n]$.  Let $I$ and $J$ be finitely generated fractional ideals of $R$. Choose a monomial ordering for $k[x_1,\hdots,x_n]$ and suppose that $\deg(IJ)\setminus (\deg(I)+\deg(J))$ is a finite set. Then
\[
\len(\T(I\otimes_RJ))\leq\tau(\deg(I),\deg(J))-|\deg(IJ)\setminus (\deg(I)+\deg(J))|.
\]
\end{prop}

\begin{proof}
Since replacing $I$ with $fI$ and $J$ with $gJ$ for any $f,g$ in $R$ does not affect the values on either side of the inequality, we may assume that $I$ and $J$ are ideals. Let $A=\deg(I)$, $B=\deg(J)$ and $S=\deg(R)$. Let $<$ denote the ordering on $\mathbb{N}_0^n$ which is associated to the ordering on the monomials of $k[x_1,\hdots,x_n]$. Note that $<$ extends naturally to an ordering on $\mathbb{Z}^n$. 
Choose $k$-bases of monic polynomials $\{f_z|\ z\in A\}$, $\{g_z|\ z\in B\}$ and $\{r_z|\ z\in S\}$ for $I$, $J$ and $R$ respectively, and write
\[ 
f_a=x^a+\sum_{i<a\in\mathbb{Z}^n}\alpha_ix^i,\quad g_b=x^b+\sum_{i<b\in\mathbb{Z}^n}\beta_ix^i\quad\textrm{and}\quad r_s=x^s+\sum_{i<s\in\mathbb{Z}^n}\epsilon_ix^i 
\]
with $\alpha_i$, $\beta_i$ and $\epsilon_i$ in $k$. For each $s$ in $S$ we have $f_{a+s}=r_sf_a+\sum_{x<a+s}\delta_xf_x$ for some $\delta_x$ in $k$. Therefore
\begin{align*}
f_{a+s}\otimes g_b&\textstyle =r_sf_a\otimes g_b+\sum_{x<a+s}\delta_xf_x\otimes g_b\\ 
&\textstyle =f_{a}\otimes r_sg_{b}+\sum_{x<a+s}\delta_xf_x\otimes g_b\\ 
&\textstyle =f_a\otimes g_{b+s}+\sum_{y<b+s}\gamma_{y}f_a\otimes g_y+\sum_{x<a+s}\delta_xf_x\otimes g_b
\end{align*}
for some $\delta_x$ and $\gamma_y$ in $k$.  It follows that for any $a,\ c$ in $A$ and $b,\ d$ in $B$ such that $a\otimes b=c\otimes d$ we have 
\[
\textstyle f_a\otimes g_b=f_c\otimes g_d+\sum_{x+y<a+b}\alpha_{xy}f_x\otimes g_y\textrm{ with }\alpha_{xy}\in k.
\]

For each $z$ in $A+B$ let $c_z=\tau_z(A,B)$, and fix $a_{z0},a_{z1},\hdots,a_{zc_z}$ in $A$ such that 
\[
\chi_{A,B}^{-1}(z)=\{a_{z0}\otimes(z-a_{z0}),\hdots, a_{zc_z}\otimes(z-a_{zc_z})\}.
\]
Let $\gamma$ be an element of  $I\otimes_R J$. Then $\gamma=\sum_{a\in A,b\in B}\alpha_{ab}f_a\otimes g_b$ for some $\alpha_{ab}$ in $k$. Starting with $z=a+b$ maximal such that $\alpha_{ab}\neq 0$ and proceeding inductively as $z$ decreases, we may replace each term $\alpha_{ab}f_a\otimes g_b$ with a sum of the form  
\[
\textstyle\alpha_{ab}f_{a_{zi}}\otimes g_{z-a_{zi}}+\sum_{x+y<z}\beta_{xy}f_x\otimes g_y,
\] so that
$\gamma=\textstyle\sum_{z\in A+B}\sum_{i=0}^{c_z}\alpha_{zi}f_{a_{zi}}\otimes g_{z-a_{zi}}$
for some $\alpha_{zi}$ in $k$.  

We may assume that there are only finitely many $z$ in $\mathbb{Z}^n$ such that $\tau_z(A,B)\neq 0$; otherwise $\tau(A,B)=\infty$ in which case the inequality in the proposition holds.

Let $S_0=\{z\in A+B|\ \tau_z(A,B)\neq 0\}$.  For each $i\geq 0$ let $M_j$ be the $k$-vector space generated by elements of the form $f_{a_{zi}}\otimes g_{z-a_{zi}}$ with $z$ in $S_j$ and let $S_{j+1}=\deg(\pi(M_j))$.  
Then $S_j\subseteq S_{j+1}$ and $M_j\subseteq M_{j+1}$ for all $j$ in $\mathbb{N}_0$.  
Note that $S_0$ is a finite set. Choose $j$ in $\mathbb{N}_0$ such that $S_h$ is a finite set for all $h\leq j$.  Since $f_{a_{zi}}g_{z-a_{zi}}$ is a finite sum of scalars times monomials for each $i$ and $z$, it follows that the set $S_{j+1}=\deg\{\sum_{z,i}\alpha_{zi}f_{a_{zi}}g_{z-a_{zi}}|\ z\in S_j\}$ is also finite; hence $M_j$ is a finite dimensional $k$-vector space for all $j$ in $\mathbb{N}_0$. Fix $j$ in $\mathbb{N}_0$ and choose 
$z_{j+1}$ in $S_{j+1}\setminus S_j$. Then $z_{j+1}=\deg(\sum_{z\in S_j, i}\beta_{zi}f_{a_{zi}}g_{z-a_{zi}})$ for some $\beta_{zi}$ in $k$.  Since $z_{j+1}$ is not in $S_j$, it follows that there exists $z_j\in S_j\setminus S_{j-1}$ such that for some $i$ we have $\beta_{z_ji}\neq 0$ and $z_j>z_{j+1}$ (set $S_{-1}=\emptyset$ for the case $j=0$). It follows that an increasing sequence $S_0\subsetneq S_1\subsetneq S_2\subsetneq \hdots$ yields a decreasing sequence $z_0>z_1>z_2>\hdots$ in $v(IJ)$.  The reason that we must get a decreasing sequence is that for every element of
 $S_j\setminus S_{j-1}$ we have a finite decreasing sequence of the form $z_0,\hdots,z_j$ and at least one of these must extend to a finite decreasing sequence of the form $z_0,\hdots,z_{j+1}$. Since $<$ satisfies the descending chain condition on $v(IJ)$, it follows that $z_0,\hdots,z_j$ cannot extend to an infinite decreasing sequence; hence the sequence 
$S_0\subsetneq S_1\subsetneq S_2\subsetneq \hdots$ stabilizes.  Choose $N$ in $\mathbb{N}_0$ such that $S_N=S_{N+1}$. 

Let $\gamma=\textstyle\sum_{z\in A+B}\sum_{i=0}^{c_z}\alpha_{zi}f_{a_{zi}}\otimes g_{z-a_{zi}}$ and suppose that there exists $y$ in $(A+B)\setminus S_N$ such that $\alpha_{y0}\neq 0$. For all $z\in (A+B)\setminus S_N$ we have $c_z:=\tau_z(A,B)=0$, as $\{z\in A+B|\ \tau_z(A,B)\neq 0\}\subseteq S_N$.  Let $y$ be maximal among all such choices.  Then we claim $\gamma$ is not in $\ker(\pi)$.  Otherwise 
\[
\textstyle\deg(\sum_{z\in A+B, z>y}\sum_{i=0}^{c_z}\alpha_{zi}f_{a_{zi}}g_{z-a_{zi}})=y\in S_N,
\]
which would be a contradiction.  Thus $\ker(\pi)\subseteq M_N$.  

Let $y$ be an element of $\deg(IJ)\setminus(\deg(I)+\deg(J))$.  Choose 
\[
\delta=\textstyle\sum_{z\in A+B}\sum_{i=0}^{c_z}\alpha_{zi}f_{a_{zi}}\otimes g_{z-a_{zi}}
\]
such that $\deg(\pi(\delta))=y$.  Assume that $y$ is not in $S_N$. Then there exists $w$ in $(A+B)\setminus S_N$ such that $\alpha_{w0}\neq 0$ with $w>y$. We may choose $w$ to be maximal.  However, since the coefficient of $x^w$ in $\pi(\delta)$ is zero, we have 
$w=\deg(\sum_{z\in A+B, z>x}\sum_{i=0}^{c_z}\alpha_{zi}f_{a_{zi}}g_{z-a_{zi}})$.  Since $w$ was chosen to be maximal this implies $w$ is in $S_N$, which is a contradiction. Thus $\deg(IJ)\setminus(\deg(I)+\deg(J))\subset S_N$. 
We have the following: 
\begin{align*}
\len(\T(I\otimes_RJ))&\stackrel{(1)}=\len(\ker(\pi))\\
&\stackrel{(2)}=\len(M_N)-\len(\pi(M_N))\\
&\stackrel{(3)}=\len(M_N)-|S_N|\\
&\textstyle\stackrel{(4)}\leq(\sum_{z\in S_N}|\chi^{-1}_{A,B}(z)|)-|S_N|\\
&\textstyle\stackrel{(5)}=(\sum_{z\in S_N\cap(A+B)}|\chi^{-1}_{A,B}(z)|)-|S_N|\\
&\textstyle\stackrel{(6)}=(\sum_{z\in S_N\cap(A+B)}(|\chi^{-1}_{A,B}(z)|-1)) + |S_N\cap (A+B)|-|S_N|\\
&\textstyle\stackrel{(7)}=(\sum_{z\in S_N\cap(A+B)}\tau_z(A,B))+|S_N\cap (A+B)|-|S_N|\\
&\textstyle\stackrel{(8)}=\tau(A,B)+|S_N\cap(A+B)|-|S_N|\\
&\textstyle\stackrel{(9)}=\tau(\deg(I),\deg(J))+|S_N\cap(\deg(I)+\deg(J))|-|S_N|\\
&\textstyle\stackrel{(10)}=\tau(\deg(I),\deg(J))-|S_N\setminus(\deg(I)+\deg(J))|\\
&\stackrel{(11)}\leq\tau(\deg(I),\deg(J))-|\deg(IJ)\setminus(\deg(I)+\deg(J))|.
\end{align*}

The first and sixth steps above are clear. The second step above follows from the inclusion $\ker(\pi)\subset M_N$.  Since $\pi(M_N)$ is a $k$-vector space, its length is just the cardinality of $\deg(\pi(M_N))$, which is $|S_N|$, and the third step follows. By construction the generators of $M_N$ (possibly non-minimal) are in one to one correspondence with $\chi^{-1}_{A,B}(S_N)$, and the fourth step follows.   
Since $|\chi^{-1}_{A,B}(z)|$ is non-zero if and only if $z$ is in $A+B$, we get the fifth and seventh steps. Since $\{z\in \mathbb{Z}|\ \tau_z(A,B)\neq 0\}\subseteq S_N\cap(A+B)$, we have the eighth step.  The ninth step simply applies the identities $A=\deg(I)$ and $B=\deg(J)$. Since $S_N$ is a finite set the tenth step follows from basic set theory. The inclusion $\deg(IJ)\setminus(\deg(I)+\deg(J))\subset S_N$ implies the last step and the result follows.
\end{proof}

\begin{lem} \label{prop:up}Let $R$ be a one-dimensional analytically irreducible residually rational ring. Let $I$ and $J$ be fractional ideals of $R$ with $f,f'$ in $I$ and $g,g'$ in $J$.  Suppose that 
$v(f)\otimes v(g)=v(f')\otimes v(g')$ in $v(I)\otimes_{v(R)}v(J)$.
Then
\[
\textstyle f\otimes g=u f'\otimes g'+\sum_{i=1}^n a_i\otimes b_i\in I\otimes_RJ
\]
where $u$ is a unit, each $a_i$ is in $I$ and each $b_i$ is in $J$ such that $v(a_ib_i)>v(fg)$.
\end{lem}

\begin{proof} 
Suppose that $v(f)=v(f')+s$ for some $s$ in $v(R)$.  
Let $r$ be in $R$ such that $v(r)=s$. By Remark \ref{prop:snug}~\eqref{unit} there are units $u$ and $u'$ such that 
$v(f-urf')>v(f)$ and $v(rg-u'g')>v(g')$. Since 
\[
f\otimes_Rg 
=uu' f'\otimes g'+ (f-urf')\otimes_Rg + uf\otimes(rg-u'g'),
\]
it follows that the result holds in this special case.  Similarly the result holds when 
$v(f)+s=v(f')$.  

In general we may choose sequences $f_1,\hdots, f_n$ and 
$g_1,\hdots, g_n$ with $f=f_1$, $g=g_1$, $f'=f_n$ and $g'=g_n$ such that for all $i$, $v(f_ig_i)=v(fg)$ and there exists an element $s_i$ in $v(R)$ where either
$v(f_i)=v(f_{i+1})+s_i$ or $v(f_i)+s_i=v(f_{i+1})$. Now the general result follows by induction from the elementary case.
\end{proof}

\begin{lem} 
\label{prop:ifCthenTF} Let $R$ be a one-dimensional analytically irreducible residually rational ring. 
Let $I$ and $J$ be fractional ideals of $R$ and let 
$\gamma=\sum_{i=1}^n f_i\otimes g_i$ be an element of $\T(I\otimes_R J)$. If 
\[
v(f_ig_i)\geq 2F+2+\min(v(I))+\min(v(J))
\] 
for all $i$, then $\gamma=0$.
\end{lem}

\begin{proof} 
Let $x$ be in $I$ and $y$ be in $J$ such that $v(x)$ and $v(y)$ are minimal. 
If $v(f_i)<F+1+v(y)$, then $v(g_i)>F+1+v(y)$ and $g_i=r_iy$ for some $r_i$ in $\mathfrak{n}^{F+1}$. By Remark \ref{prop:snug}~\eqref{snug} we have that $r_i$ is in $R$, so
$f_i\otimes g_i=f_i\otimes r_iy=r_if_i\otimes y$.  
Since $v(r_if_i)>F+1+v(x)$, we may write $\gamma$ as $\sum_{i=1}^nf_i'\otimes g_i'$ where $v(f_i')>F+1+v(x)$ for all $i$. For each $i$ there exists 
$r'_i$ in $\mathfrak{n}^{F+1}$ such that $f_i'=r'_ix$.  Thus
\[
\textstyle\gamma=\sum_{i=1}^nf_i'\otimes g_i'=\sum_{i=1}^nr'_ix\otimes g_i'=x\otimes \sum_{i=1}^nr'_ig_i'.
\]
Thus $0=\pi(\gamma)=x\sum_{i=1}^nr'_ig_i'$; hence $\gamma=x\otimes \sum_{i=1}^nr'_ig_i'=0$.
\end{proof}

\begin{prop} \label{prop:inequality} Let $R$ be a one-dimensional analytically irreducible residually rational ring. Let $I$ and $J$ be fractional ideals of $R$. Then
\[ 
\len(\T(I\otimes_RJ))\leq \tau(v(I),v(J))-|v(IJ)\diagdown(v(I)+v(J))|.
\]
\end{prop}

\begin{proof} 
For any integer $c$ let 
\[
M_c:=\{\textstyle\sum_{i=1}^na_i\otimes_Rb_i \in I\otimes_R J|\ v(a_ib_i)\geq c\}\quad\textrm{ and }\quad
N_c:=\{x\in IJ|\ v(x)\geq c\}.
\]
Let 
$z=2F+2+\min(v(I))+\min(v(J))$. Let $M:=M_z$ and $N:=N_z$. 
Clearly $\pi(M)\subseteq N$.  
Lemma \ref{prop:ifCthenTF} 
implies that $M\cap \T(I\otimes J)=0$. Therefore, $\pi_{|_{M}}$ is injective. Let $x$ be an element of  $N$. Then $x=fg$ with $f$ in $\mathfrak{n}^{F+1+v(I)}\subseteq I$ and $g$ in $\mathfrak{n}^{F+1+v(J)}\subseteq J$. Since $f\otimes g$ is in $M$ and $\pi(f\otimes g)=x$, it follows that $\pi_{|_M}:M\to N$ is an isomorphism.
Let $\overline{\pi}:(I\otimes_RJ)/M\to (IJ)/N$ be the map induced from $\pi:I\otimes_RJ\to IJ$. The next inequality follows from Lemma \ref{prop:up}.
\begin{eqnarray*}
\len(M_c/M_{c+1})&\leq &|\{a\otimes b\in v(I)\otimes_{v(R)} v(J): \ a+b=c\}|\\
&=&
\left\{
\begin{array}{ll}
\tau_c(v(I),v(J))+1 & \textrm{if }c\in v(I)+v(J)\\
0 & \textrm{if }c\notin v(I)+v(J)
\end{array}
\right.
\end{eqnarray*}
Note that $\tau_c(v(I),v(J))=0$ if $c$ is not in $v(I)+v(J)$ or $c\geq z$. Therefore
\begin{eqnarray*}
\len((I\otimes_R J)/M)&=&\textstyle\sum_{c<z}\len(M_c/M_{c+1})\\
&\leq&\textstyle\sum_{{c<z\ \ \ \ \ \ \ \ }\atop{c\in v(I)+v(J)}}\tau_c(v(I),v(J))+1\\
&=&\tau(v(I),v(J))+|(v(I)+v(J))\diagdown v(N)|.
\end{eqnarray*}
This explains the fourth step in the next sequence.
\begin{eqnarray*}
\len(\T(I\otimes J))&\stackrel{(1)}=&\len(\ker(\pi))\\
&\stackrel{(2)}=&\len(\ker(\overline{\pi}))\\
&\stackrel{(3)}=&\len((I\otimes_R J)/M)-\len((IJ)/N)\\
&\stackrel{(4)}\leq& \tau(v(I),v(J))+|(v(I)+v(J))\diagdown v(N)|-\len((IJ)/N)\\
&\stackrel{(5)}=& \tau(v(I),v(J))+|(v(I)+v(J))\diagdown v(N)|-|v(IJ)\diagdown v(N)|\\
&\stackrel{(6)}=&\tau(v(I),v(J))-|v(IJ)\diagdown(v(I)+v(J))|.
\end{eqnarray*}
The first step comes from the discussion at the beginning of Section 1. The second step follows from the fact that $\pi$ maps $M$ isomorphically onto $N$. The third step follows from surjectivity of $\overline \pi$. The fifth step is given by the equality $\len((IJ)/N)=|v(IJ)\diagdown v(N)|$ from  Remark \ref{lem:quotient}. The last step is straightforward, and the result follows. 
\end{proof}

The next example shows that the inequality in Proposition \ref{prop:inequality} can be strict.

\begin{ex}
Let $R=k[t^4,t^5,t^6]_{(t^4,t^5,t^6)}$. Let $I=(t^4,t^5)$ and $J=(t^4,t^5+t^7)$ be fractional ideals of $R$.  
Notice that $v(I)=v(J)=\{4,5,8,\to\}$. When $z\neq 9,16$ we have $\tau_z(v(I),v(J))=0$. Also $v(I)\otimes_{v(R)}v(J)$ has exactly two elements in degree $9$ and in degree $16$. Thus $\tau(v(I),v(J))=2$. Specifically 
$5\otimes 4\neq 4\otimes 5$ in $v(I)\otimes_{v(R)}v(J)$ and 
$12\otimes 4=8\otimes 8=4\otimes 12\neq 11\otimes 5=5\otimes 11$ in $v(I)\otimes_{v(R)}v(J)$.
We have $v(IJ)=\{8,\to\}$ and 
$v(I)+v(J)=\{8,9,10,12,\to\}$.  Therefore 
\[
|v(IJ)\diagdown (v(I)+v(J))|=|\{11\}|=1.
\]
It follows from Proposition \ref{prop:inequality} that 
\[
\lambda(\T(I\otimes_R J))\leq\tau(v(I),v(J))-|v(IJ)\diagdown (v(I)+v(J))|=2-1=1.
\]
In this case it turns out that the inequality is strict and $\lambda(\T(I\otimes_R J))=0$. The difference can be accounted for by the relation $t^{4}\otimes t^{12}-t^{5}\otimes t^{11}=0$.  Specifically
$t^{4}\otimes t^{12}
=t^4\otimes(t^{10}+t^{12})-t^{4}\otimes t^{10}
=t^{9}\otimes(t^5+t^7)-t^{10}\otimes t^{4}\\
\indent\indent \ \ \ =t^5\otimes(t^9+t^{11})-t^5\otimes t^9
=t^{5}\otimes t^{11}$.
%
\end{ex}

\begin{prop}  \label{prop:rank}
Let $R$ be a one-dimensional analytically irreducible residually rational ring with maximal ideal $\m=(t^{n_1},\hdots ,t^{n_e})$ and let $R'$ be a $\mathbb{Z}^n$ standard graded $k$-subalgebra of $k[x_1,\hdots,x_n]$. Let $I$ and $J$ be monomial fractional ideals of $R$. Let $I'$ and $J'$ be finitely generated monomial fractional ideals of $R'$. Then
$\len(\T(I\otimes_RJ))=\tau(v(I),v(J))$ and $\len(\T(I'\otimes_{R'}J')=\tau(\deg(I'),\deg(J'))$.
\end{prop}

\begin{proof}
Let $I=(t^{a_1},\hdots,t^{a_m})$. By Propositions  \ref{prop:MtoG} and \ref{prop:inequality} we have the inequalities
$\len(\T(I\otimes_{R}J))\leq\tau(\deg(I),\deg(J))$ and $\len(\T(I'\otimes_{R'}J'))\leq\tau(v(I'),v(J'))$.
By Theorem \ref{prop:moduleOneWay2} we have 
\[
M:=\frac{\sum_{i<j}(t^{a_i}J\cap t^{a_j}J)(\mathbf{e}_i-\mathbf{e}_j)}{\sum_{i<j}(t^{a_i}R\cap t^{a_j}R)J(\mathbf{e}_i-\mathbf{e}_j)}\cong\T(I\otimes_RJ).
\]
For each $z$ in $\deg(IJ)$ choose $h_z$ such that $z-a_{h_z}$ is an element of $\deg(J)$. Then $t^z$ is an element of  $t^{a_{h_z}}J$. If we take the quotient of $M$ by all of the distinct non-zero elements of the form $t^z(\mathbf{e_i}-\mathbf{e}_{h_z})$ starting with $z$ as large as possible and letting $z$ decrease, then at each stage the length of the quotient decreases. Thus 
$\len(M)$ is at least the number of elements of this type.
Consider the following equivalences. 
\[
\textstyle t^z(\mathbf{e}_r-\mathbf{e}_s)\in\sum_{i<j}(t^{a_i}J\cap t^{a_j}J)(\mathbf{e}_i-\mathbf{e}_j)
\iff z\in (a_r+v(J))\textrm{ and }z\in (a_s+v(J))\\
\]
\quad \ \ \ \ \ \ \ \ \ \ \ \ \ \ \ \ \ \ \ \ \ \ \  \ \ \ \ \ \ \ \ \ \ \ 
\ \ \ \ \ \ \   $\iff a_r\otimes(z-a_r),a_s\otimes(z-a_s)\in v(I)\otimes v(J)$
\\ \\
Suppose that $t^z(\mathbf{e}_r-\mathbf{e}_s)$ is an element of $\sum_{i<j}(t^{a_i}R\cap t^{a_j}R)J(\mathbf{e}_i-\mathbf{e}_j)$. Then there exists a sequence $r=i_1,i_2\hdots ,i_h=s$ such that $t^z$ is in $(t^{a_{i_j}}R\cap t^{a_{i_{j+1}}}R)J$ for $j=1,\hdots,h-1$; hence there exist $b_1,\hdots, b_{h-1}$ in $\deg(J)$ such that $z-b_j$ is an element of  $(a_{i_j}+\deg(R))\cap(a_{i_{j+1}}+\deg(R))$. Thus
\[
a_{i_j}\otimes(z-a_{i_j})=(z-b_j)\otimes b_j=a_{i_{j+1}}\otimes (z-a_{i_{j+1}})\textrm{ for }j=1,\hdots,h-1.
\]
Thus $a_r\otimes(z-a_r)$ and $a_s\otimes(z-a_s)$ represent the same element in $v(I)\otimes_{v(R)}v(J)$. It follows that for a given $z$ in $v(IJ)$ the number of distinct non-zero elements in $M$ of the form $t^z(\mathbf{e_i}-\mathbf{e}_{h_z})$ is at least 
\[
|\{a_i\otimes(z-a_i)\in v(I)\otimes v(J)\}|-1=\tau_z(v(I),v(J)).
\]
 Thus 
\[
\textstyle{\len(\T(I\otimes_RJ))}=\len(M)\geq\sum_{z\in\mathbb{Z}}\tau_z(v(I),v(J))=\tau(v(I),v(J)).
\]
A similar argument shows that ${\len(\T(I'\otimes_{R'}J'))}\geq\tau(\deg(I'),\deg(J'))$.
\end{proof}

\section{Hypersurfaces}

In this section we fix relatively prime integers $a$ and $b$ such that $b>a>1$.
Let $Z:=\mathbb{Z}^2/(b,-a)\mathbb{Z}$ be the quotient group of $\mathbb{Z}^2$. 
For any point $(x,y)$ in $\mathbb{Z}^2$ let $\overline{(x,y)}$ in $Z$ denote the coset containing the point $(x,y)$.
Let $\psi: Z\to\mathbb{Z}$ be the group isomorphism defined by $\psi(\overline{(x,y)}):=ax+by$.
Since $\psi$ is an isomorphism, it establishes an equivalence between sub-semigroups of $Z$ and sub-semigroups of $\mathbb{Z}$.  Let $S=\langle a,b\rangle$. Then $S_{Z}:=\psi^{-1}(S)$ is the sub-semigroup of $Z$ generated by 
$\psi^{-1}(0)=\overline{(0,0)}$, $\psi^{-1}(a)=\overline{(1,0)}$ and $\psi^{-1}(b)=\overline{(0,1)}$.
Given a relative ideal $A$ of $S$ we denote the relative ideal $\psi^{-1}(A)$ of $S_{Z}$ by $A_{Z}$. 

The set
\[
\B{A_Z}:=\{\overline{(x,y)}\in A_Z|\ \overline{(x-1,y-1)}\notin A_Z\}
\]
will be referred to as the {\it boundary} of $A_Z$.
The {\it Apery set} of $A$ for some $n$ in $S$ is the set $\Ap(A,n):=\{a\in A|\ a-n\notin A\}$.
Note that $\B{A_Z}=\psi^{-1}(\Ap(A,a+b))$.

\begin{ex} \label{ex:basic}
Let $S=\langle 5,7\rangle$ and
$A=(17,21,25)$. The generators of $A_Z$ are 
$\psi^{-1}(17)=\overline{(2,1)}$, $\psi^{-1}(21)=\overline{(0,3)}$ and $\psi^{-1}(25)=\overline{(5,0)}$.
We represent $Z$ on a section of the lattice in the plain which depicts $\mathbb{Z}^2$. The region depicted below extends infinitely between the parallel lines. Two points in $\mathbb{Z}^2$ are equivalent when they differ by an integer multiple of the vector $v=(7,-5)$. Every point in $\mathbb{Z}^2$ is uniquely equivalent to one of the points in the region below. We represent each element of $A_Z$ with a $\bullet$, 
each element of $S_Z\diagdown A_Z$ with a $\circ$ and each element of $Z\diagdown S_Z$ with a $"\cdot"$. 
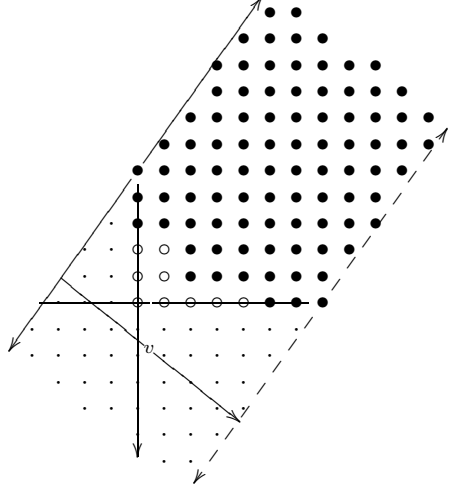
\xymatrix@=1.0em@!0{
& & & & & & & & & & & & & & & & & & & & & & & & & \\
& & & & & & & & & & & & & & & \bullet & \bullet & & & & & & & & & \\
& & & & & & & & & & & & & & \bullet & \bullet & \bullet & \bullet & & & & & & & & \\
& & & & & & & & & & & & & \bullet & \bullet & \bullet & \bullet & \bullet & \bullet & \bullet & & & & & & \\
& & & & & & & & & & & & & \bullet& \bullet & \bullet & \bullet & \bullet & \bullet & \bullet & \bullet & & & & & \\
& & & & & & & & & & & & \bullet & \bullet & \bullet & \bullet & \bullet & \bullet & \bullet & \bullet & \bullet & \bullet &  & & & \\
& & & & & & & & & & & \bullet & \bullet & \bullet & \bullet & \bullet & \bullet & \bullet & \bullet & \bullet & \bullet & \bullet & & & & \\
& & & & & & & & & & \bullet \ar@{{-}->}[lllllddddddd] \ar@{{-}->}[rrrrruuuuuuu]& \bullet & \bullet & \bullet & \bullet & \bullet & \bullet & \bullet & \bullet & \bullet & \bullet & & & & & \\
& & & & & & & & & & \bullet & \bullet & \bullet & \bullet & \bullet & \bullet & \bullet & \bullet & \bullet & \bullet & & & & & & \\
& & & & & & & & & \cdot & \bullet & \bullet & \bullet & \bullet& \bullet & \bullet & \bullet & \bullet & \bullet & \bullet & & & & & & \\
& & & & & & & & \cdot & \cdot & \circ& \circ & \bullet & \bullet & \bullet & \bullet & \bullet & \bullet & \bullet & & & & & & & \\
& & & & & & & \ar@{{-}->}(48.7,-58)|{v} & \cdot & \cdot & \circ & \circ & \bullet & \bullet & \bullet & \bullet & \bullet & \bullet & & & & & & & & \\
& & & & & & & \cdot & \cdot & \cdot & \circ \ar[dddddd] \ar@{-}[rrrrrrr] & \circ \ar@{-}[lllll]& \circ & \circ & \circ & \bullet & \bullet & \bullet \ar@{-->}[lllllddddddd] \ar@{-->}[rrrrruuuuuuu] & & & & & & & & \\
& & & & & & \cdot & \cdot & \cdot & \cdot & \cdot \ar@{-}[uuuuuu] & \cdot & \cdot & \cdot & \cdot & \cdot & \cdot & & & & & & & & & \\
& & & & & & \cdot & \cdot & \cdot & \cdot & \cdot & \cdot & \cdot & \cdot & \cdot & \cdot & & & & & & & & & & \\
& & & & & & & \cdot & \cdot & \cdot & \cdot & \cdot & \cdot & \cdot & \cdot & & & & & & & & & & & \\
& & & & & & & & \cdot & \cdot & \cdot & \cdot& \cdot & \cdot & \cdot & & & & & & & & & & & \\
& & & & & & & & & & \cdot & \cdot & \cdot & \cdot & & & & & & & & & & & & \\
& & & & & & & & & & & \cdot & \cdot & & & & & & & & & & & & & \\
& & & & & & & & & & & & & & & & & & & & & & & & & 
}
\end{ex}

\begin{notation}
Given a relative ideal $A$ of $S$. The unique minimal generating sets for  $A_Z$ and $A$ are 
$\G{A_Z}=\{\overline{(x,y)}\in A_Z|\ \overline{(x-1,y)},\ \overline{(x,y-1)}\notin A_Z\}$ and
$\G{A}=\{z\in A|\ z-a,z-b\notin A\}$ respectively.
\end{notation}

\begin{defn} \label{def:varphi}
A path $\gamma:[0,1]\to\mathbb{R}/\mathbb{Z}$ is positively oriented if there is a strictly increasing continuous map $\gamma':[0,1]\to\mathbb{R}$ such that $\gamma$ is the composition of $\gamma'$ with the natural surjection $\mathbb{R}\to\mathbb{R}/\mathbb{Z}$. We define a negatively oriented path analogously.

Let $v=(b,-a)$ be a vector and let $\varphi:Z\to\mathbb{R}/\mathbb{Z}$ be the map given by
\[
\varphi(\overline{(x,y)}):=\frac{bx-ay}{a^2+b^2}+\mathbb{Z}=\frac{(x,y)\cdot v}{|v|^2}+\mathbb{Z}.
\]
In particular $\varphi$ sends the elements of $\B{A_Z}$ (resp. $\G{A_Z}$) to distinct elements of $\mathbb{R}/\mathbb{Z}$. Therefore, the elements of $\B{A_Z}$ are cyclically ordered by the order that their images occur when traversing 
$\mathbb{R}/\mathbb{Z}$ in the positively oriented direction.

Let $p$ and $q$ be generators of $A_Z$. We say that $q$ follows $p$ and that $p$ precedes $q$ in $\G{A_Z}$ (resp. $\B{A_Z}$) if $\varphi(p)$ is the next element of $\varphi(\G{A_Z})$ to occur after $\varphi(q)$ when traversing $\mathbb{R}/\mathbb{Z}$ in the positively oriented direction.
\end{defn}
\begin{defn} \label{def:interval}
The closed interval $[p,q]_A\subseteq \B{A_Z}$ is the set of all elements in $\B{A_Z}$ that successively follow one another in $\B{A_Z}$ starting with $p$ up to and including $q$. 
Similarly we define the open interval $(p,q)_A:=[p,q]_A\diagdown\{p,q\}$, and the half open intervals $(p,q]_A:=[p,q]_A\diagdown\{p\}$ and $[p,q)_A:=[p,q]_A\diagdown\{q\}$.
\end{defn}
\begin{defn}
Let 
$\MB{A_Z}=\{\overline{(x,y)}\in\B{A_Z}|\ \overline{(x-1,y)}, \overline{(x,y-1)}\in A_Z\}$.
\end{defn}

In the next example one should reference the diagram in Example \ref{ex:basic} for clarity.

\begin{ex}
Let  $S=\langle 5,7\rangle$ and
$A=(17,21,25)$ be the same as in Example \ref{ex:basic}.  Then the cyclically ordered elements of $\G{A_Z}$ are $\overline{(0,3)}$, $\overline{(2,1)}$ and $\overline{(5,0)}$. The set of maximal elements of the boundary is
$\MB{A_Z}=\{\overline{(0,5)}, \overline{(2,3)}, \overline{(5,1)}\}$.
The cyclically ordered elements of the boundary $\B{A_Z}$ are 
\[
\overline{(0,5)},\overline{(0,4)},\overline{(0,3)},\overline{(1,3)},\overline{(2,3)},\overline{(2,2)},\overline{(2,1)},\overline{(3,1)},\overline{(4,1)},\overline{(5,1)},\overline{(5,0)}\textrm{ and }\overline{(6,0)}.
\]
Notice that we do not mention $\overline{(7,0)}$, since $\overline{(7,0)}=\overline{(0,5)}$. Also we could have chosen to begin this list with any of the elements since the ordering is cyclic. Lastly the interval 
$[\overline{(5,1)},\overline{(1,3)}]_A=\{\overline{(5,1)},\overline{(5,0)},\overline{(6,0)},\overline{(0,5)},\overline{(0,4)},\overline{(0,3)},\overline{(1,3)}\}$.
\end{ex}

\begin{lem}\label{lem:ordering}
Let $A$ be a relative ideal of $S=\langle a,b\rangle$.  Let $\GZ{A_Z}=\{p_1,\hdots ,p_n\}$ such that 
$p_{i+1}$ follows $p_i$ in $\G{A_Z}$ for all $i$ in $\mathbb{Z}/n\mathbb{Z}$. Then we may choose ordered integers 
$x_1<x_2<\hdots <x_n<x_1+b$ and $y_1>\hdots >y_n>y_1-a$ such that $p_i=\overline{(x_i,y_i)}$.
\end{lem}

\begin{proof}
Choose integers $x_1$ and $y_1$ such that $p_1=\overline{(x_1,y_1)}$. 
For each generator $p_i$ of $A_Z$ choose integers $x'_i$ and $y'_i$ such that $p_i=\overline{(x'_i,y'_i)}$.  Then  
$p_i=\overline{(x'_i+nb,y'_i-na)}$ for any integer $n$.  There exists a unique integer $n_i$ such that 
$x_1\leq x'_i+n_ib<x_1+b$.  Let $x_i=x'_i+n_ib$ and $y_i=y'_i-n_ia$.
For some pair $i,j$ if $x_i=x_j$, then it follows that either $p_i$ is in $p_j+S_Z$ or $p_j$ is in $p_i+S_Z$.  Since the $p_i$ are distinct generators, it follows that in this case $i=j$;  hence the $x_i$ are unique. Now permute the labels $x_2,x_3,\hdots,x_n$ and apply the same permutation to the labels $y_2,y_3,\hdots,y_n$ so that 
\[
x_1<x_2<\hdots <x_n<x_1+b.
\]

For $i=2,\hdots, n$ we have $x_i>x_{i-1}$ and $\overline{(x_i,y_i)}$ is not in $\overline{(x_{i-1},y_{i-1})}+A_Z$. Therefore 
$y_i<y_{i-1}$. Assume that $y_n\leq y_1-a$. Then there exists a positive integer $\ell$ such that 
$y_1-a<y_n+\ell a\leq y_1$.
It follows that $\overline{(x_1,y_1)}$ is in $\overline{(x_n-\ell b,y_n+\ell a)}+S_Z=\overline{(x_n,y_n)}+S_Z$. This contradicts the minimality of the generators; hence 
\[
y_1>y_2>\hdots >y_n>y_1-a.
\] 

 It follows from the ordering on the elements $x_i$ and the $y_i$ that 
\[
\frac{x_1b-y_1a}{a^2+b^2}<\cdots <\frac{x_nb-y_na}{a^2+b^2}<\frac{(x_1+a)b-(y_1-b)a}{a^2+b^2}=\frac{x_1b-y_1a}{a^2+b^2}+1.
\]
Therefore, $\overline{(x_1,y_1)}, \hdots ,\overline{(x_n,y_n)}$ occur one after another in the cyclic ordering on $\G{A_Z}$. Thus the permutation we applied to the labels $x_2,\hdots,x_n$ was the trivial permutation, and the result follows.
\end{proof}

\begin{thm} \label{thm:semigroup}
Let $A$ and $B$ be non-principal relative ideals of $S=\langle a,b\rangle$.  Then, 
\[
\textstyle\tau(A,B)+|\{z\in\mathbb{Z}|\ \tau_z(A,B)\neq 0\}|\geq\mu(A)\mu(B),
\]
and it follows that $\textstyle\tau(A,B)\geq \frac{1}{2}\mu(A)\mu(B)$.
\end{thm} 

We will postpone the proof of Theorem \ref{thm:semigroup} until the end of the paper.

\begin{cor} \label{thm:main}
Suppose that $R=k[x^a,x^b]$ and that $R'$ is a one-dimensional analytically irreducible residually rational ring with maximal ideal $\m=(t^a,t^b)$ for some $t$ in $\overline{R}$. Let $I$ and $J$ be monomial ideals of $R$. Let $I'$ and $J'$ be monomial ideals of $R'$. Then $\len(\T(I\otimes_{R} J))\geq \frac{1}{2}\mu(I)\mu(J)$ and $\len(\T(I'\otimes_{R'}J'))\geq\frac{1}{2}\mu(I')\mu(J')$.
\end{cor}

\begin{proof}
Apply Proposition \ref{prop:rank} and Theorem \ref{thm:semigroup}.
\end{proof}

\section{The inverse of an ideal}

\begin{defn} Let $S$ be a numerical semigroup with relative ideal $A$.
The inverse of $A$ is the relative ideal
$A^*:=\{z\in\mathbb{Z}|\ z+A\subseteq S\}$.
\end{defn} 

\begin{remark}
Let $R$ be a one-dimensional analytically irreducible residually rational ring and let $R'$ be a $k$ subalgebra of $k[x]$. Let $I$ and $J$ be fractional deals of $R$ and $R'$ respectively.
If $f$ is in $I^{-1}$, then $fI\subseteq R$ implies $v(f)+v(I)=v(fI)\subseteq v(R)$.  Therefore $v(f)$ is in $v(I)^*$ and $v(I^{-1})\subseteq v(I)^*$. Similarly $\deg(J^{-1})\subseteq\deg(J)^*$.  

If $(t^{a_1},\hdots,t^{a_n})$ is the maximal ideal of $R$ for some $t$ in $\overline{R}$ and $R'$ is a standard graded $k$-subalgebra of $k[x]$, then $v(I^{-1})=v(I)^*$ and $\deg(J^{-1})=\deg(J)^*$.

However, equality does not necessarily hold in the general setting. For instance
if $R=k[[t^5,t^7,t^9]]$ and $I=(t^5,t^7+t^8)$, then $5$ is in $v(I)^*$ and 
$5$ is not in $v(I^{-1})$.
\end{remark}

\begin{remark} \label{prop:monomialDual}
Let $R$ be an analytically irreducible residually rational ring with maximal ideal $\m=(t^{n_1},\hdots ,t^{n_e})$.
If $I$ is a monomial fractional ideal, then so is $I^{-1}$. If $J$ is  a reflexive fractional ideal, then $J$ is monomial if and only if $J^{-1}$ is monomial.

By Remark \ref{prop:snug}~\eqref{snug}  we have that $z$ is in $v(R)$ if and only if $t^z$ is in $R$. Therefore $fI\subseteq R$ implies $t^{v(f)}I\subseteq R$; hence if $z$ is in $v(I^{-1})$, then $t^z$ is in $I^{-1}$.  Let $I'=(t^z|\ z\in v(I^{-1}))$.  By Remark
 \ref{lem:quotient} we have $\len(I^{-1}/I')=|v(I^{-1})\setminus v(I')|=0$.  Thus $I^{-1}=I'$ is monomial. When $J$ is reflexive we get the other implication, since $(J^{-1})^{-1}=J$. 
\end{remark}

Let $s$ be an integer and let $\Gamma$ be a numerical semigroup. An \emph{arithmetic-sequence} over $\Gamma$ is a sequence of the form $(x,x+s,x+2s,\hdots,x+ns)$ such that 
$x+is$ is in $\Gamma$ for $i=0,\hdots,n$ with $n>0$.  In this case we say that the arithmetic sequence has $n$ steps. Arithmetic-sequences over $\Gamma$ with step size $s$ form a semigroup. Given arithmetic-sequences $(y,y+s,\hdots, y+as)$ and 
$(z,z+s,\hdots ,z+bs)$ over $\Gamma$, their sum $(y,\hdots,y+as)+(z,\hdots, z+bs):=(y+z,y+z+s,\hdots,y+(a+b)s)$ is also an arithmetic-sequence over $\Gamma$. We will say that an arithmetic-sequence is irreducible when it does not factor as the sum of two arithmetic-sequences. The following result is stated for semigroup rings  but an almost identical proof yields a similar result for two-generated monomial ideals over analytically irreducible residually rational rings and also for two-generated monomial ideals over a standard graded $k$-subalgebra of $k[x_1,\hdots,x_n]$.

\begin{prop} \label{prop:arithmetic}
Let $\Gamma$ be a numerical semigroup.  Let $I\cong(1,x^s)$ be a fractional ideal of $k[\Gamma]$, for some field $k$ and $s$ in $\mathbb{N}$.  Then the length of $\T(I\otimes_{k[\Gamma]} I^*))$ is equal to the number of irreducible 
arithmetic-sequences in $\Gamma$ of the form $(x,x+s,x+2s)$.     
\end{prop}

\begin{proof}
By Lemma \ref{lem:iso2gen}, $\T(I\otimes_{k[\Gamma]}I^*)\cong \frac{(I^2)^{-1}}{(I^{-1})^2}$. Therefore by Remark \ref{rem:quotient} we have 
$\len(\T(I\otimes_{k[\Gamma]}I^*))=|\deg((I^2)^-1)\setminus\deg(I^{-1})^2)|$. Note that $\deg(I^{-1})=(0,s)^*$ 
is the set of $z$ such that $(z,z+s)$ is an arithmetic-sequence in $\Gamma$. Also $\deg((I^2)^{-1})=(0,s,2s)^*$ is the 
set of $z$ such that $(z,z+s,z+2s)$ is an arithmetic sequence in $\Gamma$.   We have that 
$\deg((I^{-1})^2)=(0,s)^*+(0,s)^*$ is the set of sums $y+z$ such that $(y,y+s)$ and $(z,z+s)$ are arithmetic-
sequence in $\Gamma$. Thus $\deg((I^2)^{-1})\setminus\deg((I^{-1})^2)$ is the set of $x$ such that $(x,x+s,x+2s)$ is an irreducible sequence in $\Gamma$, and the result follows.
\end{proof}

\begin{lem} \label{prop:dim}
Let $R$ be a one-dimensional Gorenstein domain. Let $I$ and $J$ be fractional ideals of $R$ with $I\subseteq J$. Then 
$\len(I^{-1}/ J^{-1})=\len(J/ I)$.
\end{lem}

\begin{proof}
Applying $\Hom{-}{R}$ to the sequence 
$0\to I\to J\to J/ I\to 0$ 
we get the exact sequence
$(J/ I)^*\to J^*\to I^*\to \Ext{1}{J/ I}{R}\to 0$.
Since $J/ I$ is torsion and $R$ is torsion-free, $(J/ I)^*=0$; hence 
$I^{-1}/J^{-1}\cong I^*/ J^*\cong\Ext{1}{J/ I}{R}$. 
Let $E=\bigoplus_{\m\in\mspec(R)}\E_R(R/\m)$ where $\E_R(R/\m)$ is the injective hull of $R/\m$. 
Then $R\to K\to E\to 0$ is a minimal injective resolution of $R$. Applying 
$\Hom{R/I}{-}$ we get the exact sequence 
\[
\Hom{J/I}{K}\to\Hom{J/I}{E}\to\Ext{1}{J/I}{R}\to\Ext{1}{J/I}{K}.
\]
Since $K$ is injective, $\Ext{1}{J/I}{K}=0$. Since $J/I$ is torsion, $\Hom{J/I}{K}=0$. Thus $\Hom{J/I}{E}\cong \Ext{1}{J/I}{R}\cong I^{-1}/ J^{-1}$. Since Matlis duality preserves length, the result follows.
\end{proof}

\begin{defn} 
A numerical semigroup $S$ is {\it symmetric} when
$S=\{z|\ F-z\notin S\}$.
\end{defn}

In \cite{Kunz} E. Kunz showed that a one-dimensional analytically irreducible residually rational ring $R$ is Gorenstein if and only if the semigroup $v(R)$ is symmetric. The following result enriches this theory by extending the result to ideals.

\begin{prop} \label{prop:reflect}
Let $R$ be a one-dimensional analytically irreducible residually rational Gorenstein ring. Let $I$ be a fractional ideal of $R$. Then 
\[
v(I)^*=v(I^{-1})=\{z|\ F-z\notin v(I)\}.
\]
\end{prop}

\begin{proof} 
Let $x$ be an element of $I^{-1}\diagdown\{0\}$. Then  $J:=xI\subset R$.  
Let $z$ be in $v(J^{-1})$. Then there exists $f$ in $J^{-1}$ such that $v(f)=z$.  Let $g$ be an element of $K$ such that $v(g)=F-z$. Since $v(fg)=F$, it follows that $fg$ is not in $R$. Since $g$ is not in $J$ whenever $v(g)=F-z$ we have that $F-z$ is not in $v(J)$. Thus $v(J^{-1})\subseteq\{z|\ F-z\notin v(J)\}$.

This implies the inclusion $\{F-z\in v(R)|\ z\in v(J^{-1})\}\subseteq v(R)\diagdown v(J)$, which explains the fifth step below.
\begin{eqnarray*}
\len(R/ J)&\stackrel{(1)}=&\len(J^{-1}/ R)\\ 
&\stackrel{(2)}=&|v(J^{-1})\diagdown v(R)|\\
&\stackrel{(3)}=&|\{F-z\in\mathbb{Z}|\ z\in v(J^{-1})\textrm{ and }z\notin v(R)\}|\\
&\stackrel{(4)}=&|\{F-z\in v(R)|\ z\in v(J^{-1})\}| \\
&\stackrel{(5)}\leq &|v(R)\diagdown v(J)|\\
&\stackrel{(6)}=&\len(R/ J)
\end{eqnarray*}
The first step is given by Lemma \ref{prop:dim}. The second and last equalities are from Remark \ref{lem:quotient}. The third step is elementary, and the fourth equality is from \cite[Theorem]{Kunz}.

It follows that $v(R)\setminus v(J)=\{F-z\in v(R)|\ z\in v(J^{-1})\}$. If $F-z$ is not an element of $v(R)$, then $F-z$ is not in $v(J)$ and $z$ is in $v(R)\subseteq v(J^{-1})$. Thus 
$\mathbb{Z}\setminus v(J)=\{F-z\in\mathbb{Z}|\ z\in v(J^{-1})\}$, and it follows that $v(J^{-1})=\{z|\ F-z\notin v(J)\}$. As $v(I)=v(J)-v(x)$ and $v(I^{-1})=v(J^{-1})+v(x)$, we have 
\[
v(I^{-1})=\{z|\ F-z\notin v(I)\}.
\]

There exists a monomial ideal $I'$ over a one-dimensional analytically irreducible residually rational ring $R'\subseteq \overline R$ such that $v(R)=v(R')$ and $v(I)=v(I')$.  Thus $v(I)^*=v(I')^*=v(I'^{-1})=\{z|\ F-z\notin v(I')\}=\{z|\ F-z\notin v(I)\}=v(I^{-1})$. 
\end{proof}

%

\begin{thm} \label{thm:inverse}
Let $S=\langle a,b\rangle$ be a hypersurface numerical semigroup and let $A$ be a relative ideal of $S$.  Let $\GZ{A_Z}=\{p_1,\hdots ,p_n\}$ such that 
$p_{i+1}$ follows $p_i$ in $\G{A_Z}$ for all $i$ in $\mathbb{Z}/n\mathbb{Z}$. Choose integers 
$x_1<x_2<\hdots <x_n<x_1+b$ and $y_1>y_2>\hdots >y_n>y_1-a$ such that 
$p_i=\overline{(x_i,y_i)}$. Then $A^*=(-ax_1-by_n, ab-ax_{i+1}-by_{i}|\ i=1,\hdots,n-1)$.
\end{thm}
\begin{proof} The diagram below illustrates points in $\mathbb{Z}^2$ that represent 
elements of $A_Z$. 

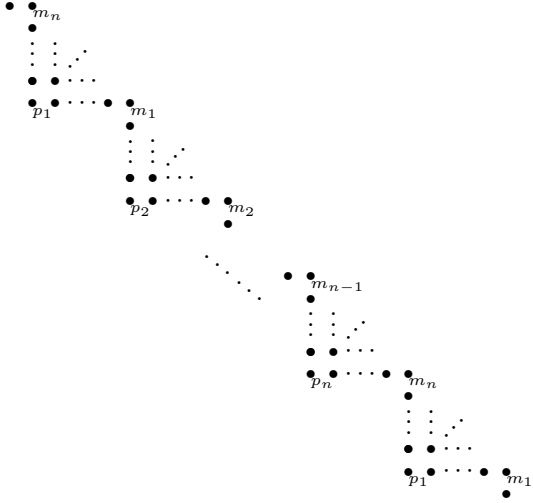
\setlength{\unitlength}{.5mm}
\begin{picture}(140,140) 
\put(-1, 134){\circle*{2}}
\put(5, 134){\circle*{2}}
\put(5,131){\tiny{$m_n$}}
\put(5,128){\circle*{2}}
\put(4,118){$\vdots$}
\put(10,118){$\vdots$}
\put(5,114){\circle*{2}}
\put(11,114){\circle*{2}}
\multiput(14,116)(2,2){3}
{$\cdot$}
\put(5,114){\circle*{2}}
\put(5,108){\circle*{2}}
\put(5,105){\tiny{$p_1$}}
\put(11,108){\circle*{2}}
\put(14,106.5){$\cdots$}
\put(14,112.5){$\cdots$}
\put(25, 108){\circle*{2}}
\put(31, 108){\circle*{2}}

\put(31,105){\tiny{$m_1$}}
\put(31,102){\circle*{2}}
\put(30,92){$\vdots$}
\put(36,92){$\vdots$}
\put(31,88){\circle*{2}}
\put(31,88){\circle*{2}}
\put(37,88){\circle*{2}}
\multiput(40,90)(2,2){3}
{$\cdot$}
\put(31,82){\circle*{2}}
\put(31,79){\tiny{$p_2$}} 
\put(37,82){\circle*{2}}
\put(40,80.5){$\cdots$}
\put(40,86.5){$\cdots$}
\put(51, 82){\circle*{2}}
\put(57, 82){\circle*{2}}
\put(57,79){\tiny{$m_2$}}
\put(57, 76){\circle*{2}}

\put(50,62){$\ddots$}
\put(58.5,55){$\ddots$}

\put(73, 62){\circle*{2}}
\put(79, 62){\circle*{2}}
\put(79,59){\tiny{$m_{n-1}$}}
\put(79,56){\circle*{2}}
\put(78,46){$\vdots$}
\put(84,46){$\vdots$}
\put(79,42){\circle*{2}}
\put(85,42){\circle*{2}}
\multiput(88,44)(2,2){3}
{$\cdot$}
\put(79,42){\circle*{2}}
\put(79,36){\circle*{2}}
\put(79,33){\tiny{$p_n$}}
\put(85,36){\circle*{2}}
\put(88,34.5){$\cdots$}
\put(88,40.5){$\cdots$}
\put(99, 36){\circle*{2}}
\put(105, 36){\circle*{2}}

\put(105,33){\tiny{$m_n$}}
\put(105,30){\circle*{2}}
\put(104,20){$\vdots$}
\put(110,20){$\vdots$}
\put(105,16){\circle*{2}}
\put(105,16){\circle*{2}}
\put(111,16){\circle*{2}}
\multiput(114,18)(2,2){3}
{$\cdot$}
\put(105,10){\circle*{2}}
\put(105,7){\tiny{$p_1$}} 
\put(111,10){\circle*{2}}
\put(114,8.5){$\cdots$}
\put(114,14.5){$\cdots$}
\put(125, 10){\circle*{2}}
\put(131, 10){\circle*{2}}
\put(131,7){\tiny{$m_1$}}
\put(131, 4){\circle*{2}}
\end{picture}\\
The coset representatives of $p_i$ are labeled by $p_i$. Similarly the elements of $\mathbb{Z}^2$ labeled by $m_1,\hdots ,m_n$ correspond to the elements of $\MB{A_Z}$. 

Let $(x_1,y_1)$ be a point in $\mathbb{Z}^2$ such that $p_1=\overline{(x_1,y_1)}$. Consider the rectangle whose corners are $(x_1,y_1)$ and $(x_1+b,y_1-a)$, which are adjacent coset representatives of $p_1$. 
Then for each $i$ in $\{2,\hdots , n\}$ there is a unique point $(x_i,y_i)$ inside the rectangle that is a coset representative $p_i$. The ordering on the integers $x_i$ and $y_i$ follows from Lemma \ref{lem:ordering}.
For all $i$ in $\mathbb{Z}/n\mathbb{Z}$ any point labeled $m_i$ has the same $x$-position as the point labeled $p_{i+1}$ below it and the same $y$-position as the point labeled $p_i$ to its left.  It follows that 
$m_i=\overline{(x_{i+1},y_i)}$ for $i=1,\hdots ,n-1$ and that $m_n=\overline{(x_1+b,y_n)}$.  

Note that $F=ab-a-b$ and $\psi^{-1}(F)=\overline{(b-1,-1)}$. 
By Proposition \ref{prop:reflect}, $A^*=\{F-z\in\mathbb{Z}|\ z\notin A\}$; hence
$A^*_Z=\{\overline{(b-1,-1)}-p\in Z|\ p\notin A_Z\}$.
The generators of $A^*_Z$ are the elements $\overline{(b-1,-1)}-p$ such that $p$ is not in $A_Z$ and $p+\overline{(0,1)}$ and $p+\overline{(1,0)}$ are in $A_Z$. These are elements of the form $\overline{(b-1,-1)}-p$ such that $p=m_i-\overline{(1,1)}$.  Thus $\GZ{A^*_Z}$ consists of the elements 
$\overline{(b-1,-1)}-(m_i-\overline{(1,1)})=\overline{(b,0)}-m_i$ for $i=1,\hdots ,n$. For $i=1,\hdots,n-1$ we have 
$\psi(\overline{(b,0)}-m_i)=\psi\overline{(b-x_{i+1},-y_i)}=ab-ax_{i+1}-by_i$ and $\psi(\overline{(b,0)}-m_n)=\psi\overline{(-x_1,-y_n)}=-ax_1-by_n$.
\end{proof}

\begin{prop}\label{prop:numgen}
Suppose that $R=k[x^a,x^b]$ or that $R$ is a one-dimensional analytically irreducible residually rational ring with maximal ideal $\m=(t^a,t^b)$. Let $I$ be monomial ideal of $R$.  Then $I\otimes_R I^*$ has a minimal generating set such that at least $2\mu(I)-2$ of the generators are torsion elements. 
\end{prop}

\begin{proof}
Let $A$ equal $v(I)$ or $\deg(I)$ depending on our choice of $R$. By Theorem \ref{thm:inverse} we have 
$A^*=(-ax_1-by_n, ab-ax_i-by_{i-1}|\ i=2,\hdots,n)$
for some integers $x_1<\hdots<x_n<x_1+b$ and $y_1>\hdots>y_n>y_1-a$. Notice that
$(-ax_1-by_n)+(ax_1+by_1)=b(y_1-y_n)$ and $(ab-ax_i-by_{i-1})+(ax_i+by_i)=b(a-(y_{i-1}-y_i))$. 
The fractional ideal $(b(y_1-y_n), b(a-(y_{i-1}-y_i))|\ i=2,\hdots,n)$ is principal.
Similarly  $(-ax_1-by_n)+(ax_n+by_n)=a(x_n-x_1)$, 
$(ab-ax_i-by_{i-1})+(ax_{i-1}+by_{i-1})=a(b-(x_{i}-x_{i-1}))$ and 
 $(a(x_n-x_1),a(b-(x_{i}-x_{i-1}))|\ i=2,\hdots,n)$ is principal. It follows that 
there exist at least $2n-2$ elements of $\G{A}+\G{A^*}$, which are not in $\G{A+A^*}$. Consequently, at least $2n-2$ of the generators of $I\otimes_RI^{-1}$ are in the kernel of the map
$\pi:I\otimes_RI^{-1}\to II^{-1}$. Since $\ker(\pi)=\T(I\otimes_RI^{-1})$, the result follows.
\end{proof}

\begin{question}
Suppose that $R$ is a one-dimensional hypersurface domain and that $M$ is rank $r$ torsion-free $R$-module.  Considering Propostion \ref{prop:numgen} and the fact that $\T(M\otimes_RM^*)\neq 0$ when $M$ is not free, we are led to the following question. Does the inequality $\mu(\T(M\otimes_RM^*)\geq 2(\mu(M)-r)$ hold in general? A positive answer in the case where $M$ is graded or simply an ideal would also be interesting.
\end{question}

\begin{proof}[Proof of Theorem \ref{thm:semigroup}] Fix generating sets $\G{A}=\{a_1,\hdots ,a_m\}$ and $\GZ{A_Z}=\{p_1,\hdots ,p_m\}$
such that $\psi(p_i)=a_i$  and $p_{i+1}$ follows $p_i$ in $\GZ{A_Z}$ for all $i$ in $\mathbb{Z}/m\mathbb{Z}$.
Also let $\G{B}=\{b_1,\hdots ,b_n\}$ and $\GZ{B_Z}=\{q_1,\hdots ,q_n\}$
such that $\psi(q_j)=b_j$ and $q_{j+1}$ follows $q_j$ in $\GZ{B_Z}$ for all $j$ in $\mathbb{Z}/n\mathbb{Z}$.   

We define a function $\delta:\mathbb{Z}/m\mathbb{Z}\times\mathbb{Z}/n\mathbb{Z}\to A\otimes_S B$ and show that it is injective.  Since 
$\delta$ is injective, we have 
\[
|\delta^{-1}(\chi^{-1}(z))|\leq|\chi^{-1}(z)|\leq \tau_z(A,B)+1.
\]
Furthermore we show that for all $c\otimes d$ in $\im(\delta)$ there exist distinct elements $e\otimes f$ and $c\otimes d$ in $A\otimes_SB$ with $c+d=e+f$. Therefore for all $z\in\im(\delta)$ we have $\tau_z(A,B)\neq 0$. Let $H=\{z\in\mathbb{Z}|\ \tau_z(A,B)\neq 0\}$.
From this we deduce the statement of the Theorem:
\[
\textstyle\tau(A,B)+|H|=\sum_{z\in H}\tau_z(A,B)+1\geq \sum_{z\in H}|\delta^{-1}(\chi^{-1}(z))|= \mu(A)\mu(B).
\]
Since $\tau(A,B)\geq |H|$, it follows that $\tau(A,B)\geq\frac{1}{2}\mu(A)\mu(B)$.

In each of the following six cases we assume conditions on elements $(i,j)$ in 
$\mathbb{Z}/m\mathbb{Z}\times\mathbb{Z}/n\mathbb{Z}$. In each case we assume that none of the previous cases occurs. 

{\bf Case 1}: Suppose $a_i+b_j$ is not an element of $\G{A+B}$. Then there exists 
$u\neq i$ and $v\neq j$ such that $a_i+b_j=a_u+b_v+s$ 
for some $s$ in $S$.  Since $a_i$ and $b_j$ are minimal generators, it follows that $\delta(i,j):=a_i\otimes b_j\neq a_u\otimes (b_v+s)$ in $A\otimes_SB$.

For the remainder of the proof suppose that $a_i+b_j$ is an element of $\G{A+B}$. Let $C$ be the relative ideal generated by $\G{A+B}\diagdown\{a_i+b_j\}$.

{\bf Case 2:} Suppose there exists $(u,v)\neq(i,j)$ such that $a_u+b_v$ is not in $C$. In this case 
$a_u+b_v=a_i+b_j+s$ for some $s$ in $S$.  Let $\delta(i,j)=a_i\otimes(b_j+s)$. Since $a_u$ and $b_v$ are minimal generators, we have that $a_u\otimes b_v\neq a_i\otimes (b_j+s)$ in $A\otimes_SB$.


Fix $g$ and $h$ such that $p_g+q_h$ precedes $p_i+q_j$ in $\G{A_Z+B_Z}$.

{\bf Case 3:} Suppose $g\neq i$ and $h\neq j$. Choose integers $x_1$, $x_2$, $y_1$ and $y_2$ with $x_1<x_2<x_1+b$ and $y_1>y_2>y_1-a$, such that $p_g+q_h=\overline{(x_1,y_1)}$ and $p_i+q_j=\overline{(x_2,y_2)}$.  Then 
$m:=\overline{(x_2,y_1)}$ is in $\MB{A_Z+B_Z}$.  Let $c=x_2-x_1$ and $d=y_1-y_2$.  Then 
$\psi(m)=a_g+b_h+ca=a_i+b_j+db$.

The following graph illustrates the relative positions of the elements $(x_1,y_1)$ and $(x_2,y_2)$ in the plain. Each point is labeled both by its coordinates and by the element of $Z$ that it represents. 

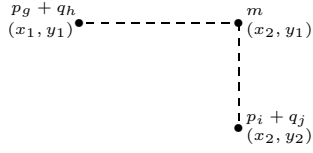
\setlength{\unitlength}{.5mm}
\begin{picture}(90,40) 
\put(25, 33){\circle*{2}}
\put(6,30){\tiny{$(x_1,y_1)$}}
\put(7,36){\tiny{$p_g+q_h$}}
\multiput(27,33)(4,0){10} 
{\line(1,0){2}}
\put(67.5,33){\circle*{2}}
\put(69.5,30){\tiny{$(x_{2},y_{1})$}}
\put(69.5,35){\tiny{$m$}}

\put(67.5,5){\circle*{2}}
\put(69.5,2){\tiny{$(x_{2},y_{2})$}}
\put(69.5,7){\tiny{$p_i+q_{j}$}}

\multiput(67.5,7)(0,4){7} 
{\line(0,1){2}}
\end{picture}

We wish to show that $a_g\otimes (b_h+ca)\neq a_i\otimes (b_j+db)$ in $A\otimes_S B$.
We claim that none of the elements of the form $p_u+q_h$ with $u\neq g$ or $p_g+q_v$ with $v\neq h$ are on the interval $[p_g+q_h,p_i+q_j]_{A+B}$; see Definition \ref{def:interval}.  None of these elements are on the interval $(m,p_i+q_j]_{A+B}$ because otherwise we would be in Case 2. By the minimality of the generating sets, it follows that 
$p_u$ is not in $p_g+S_Z$ and $q_v$ is not in $q_h+S_Z$.  Therefore, 
$p_u+q_h$ and $p_g+q_v$ are not in $p_g+q_h+S_Z$.  It follows that none of the elements in question are on the interval $[p_g+q_h,m]_{A+B}$.  From this we conclude that 
$a_g\otimes (b_h+ca)\neq a_u\otimes (b_h+s)$ and $a_g\otimes (b_h+ca)\neq a_g\otimes (b_v+s)$ for any $u\neq g$, $v\neq h$ and $s$ in $S$.
Consequently, 
\[
\textstyle\delta(i,j):=a_i\otimes (b_j+db)\neq a_g\otimes (b_h+ca)\in A\otimes_S B.
\]

Fix $\ell$ and $r$ such that $p_{\ell}+q_r$ follows $p_i+q_j$ in $\G{A_Z+B_Z}$.

{\bf Case 4:} Suppose that $\ell\neq i$ and $r\neq j$. Then there exist positive integers $c$ and $d$ simultaneously minimal such that $a_i+b_j+ca=a_{\ell}+b_r+db$. The proof that 
\[
\delta(i,j):=a_i\otimes (b_j+ca)\neq a_{\ell}\otimes(b_r+db)\in A\otimes_S B
\]
mirrors the argument in Case 3.

{\bf Case 5:} Suppose $g=i$. 
Since $A+B$ is not principle, it follows that $h\neq j$. We will show that 
$h=j-1$. The elements $p_i+q_h$, $p_i+q_{j-1}$ and $p_i+q_j$ are minimal generators of $p_i+B_Z$. Since $q_{j-1}$ precedes $q_{j}$ in $\G{A_Z}$, it follows that $p_i+q_{j-1}$ precedes $p_i+q_j$ in $\G{p_i+B_Z}$. Applying Lemma \ref{lem:ordering} to the generators of $p_i+B_Z$ we may choose integers $x_h\leq x_{j-1}<x_j<x_h+b$ and $y_h\geq y_{j-1}>y_j>y_h-a$ such that 
$p_i+q_h=\overline{(x_h,y_h)}$, $p_i+q_{j-1}=\overline{(x_{j-1},y_{j-1})}$ and $p_i+q_j=\overline{(x_j,y_j)}$. 
However, since $p_i+q_h$ precedes $p_i+q_j$ in $\G{A_Z+B_Z}$, Lemma \ref{lem:ordering} implies that  $p_i+q_{j-1}$ is not in $\alpha+S$ for any $\alpha$ in $\G{A_Z+B_Z}\setminus\{p_i+q_h\}$; hence $p_i+q_{j-1}$ is an element of  $p_i+q_h+S$ and $h=j-1$.

{\bf Case 5.1:} Suppose that $p_{i+1}+q_{j-1}$ is an element of  $p_i+q_j+S_Z$. Then there exists 
$s$ in $S$ such that $a_i+b_j+s=a_{i+1}+b_{j-1}$. Since $a_{i+1}$ and $b_{j-1}$ are both minimal generators, it follows that 
\[
\delta(i,j):=a_i\otimes(b_j+s)\neq a_{i+1}\otimes b_{j-1}\in A\otimes_SB.
\]

The graph below indicates the relative positions of the points involved in this case. In general a coset representative of $p_{i+1}+q_{j-1}$ will occur at some point in the enclosed region below, which excludes the dashed line.\\ 
\setlength{\unitlength}{.5mm}
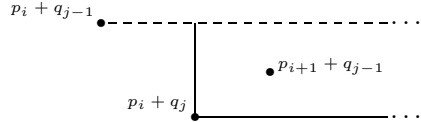
\begin{picture}(90,37) 
\put(25, 28){\circle*{2}}
\put(1,31){\tiny{$p_i+q_{j-1}$}}
\multiput(27,28)(4,0){19} 
{\line(1,0){2}}
\put(102,27.7){$\hdots$}
\put(50,3){\circle*{2}}
\put(32,6){\tiny{$p_i+q_{j}$}}
\put(50,3){\line(1,0){51}}
\put(102,2.9){$\hdots$}
\put(50,3) {\line(0,1){25}}
\put(70,15){\circle*{2}}
\put(72,16){\tiny{$p_{i+1}+q_{j-1}$}}
\end{picture}

{\bf Case 5.2:} Suppose $p_{i+1}+q_{j-1}$ is not in $p_i+q_j+S_Z$.
Consider the relative ideal
\[
D_Z:=(p_i+q_{j-1},\ p_i+q_j,\ p_{i+1}+q_v|\ v\in\mathbb{Z}/n\mathbb{Z}\ )
\]
of $S_Z$. We have that $p_i+q_{j-1}$, $p_i+q_j$ and $p_{i+1}+q_{j-1}$  are elements of  $\G{D_Z}$ such that $p_i+q_j$ is in $[p_{i}+q_{j-1},p_{i+1}+q_{j-1}]_D$. Choose $z$ in $\mathbb{Z}/n\mathbb{Z}$ such that $p_{i+1}+q_z$ is the first element of the form 
$p_{i+1}+q_v$ with $v$ in $\mathbb{Z}/n\mathbb{Z}$ to occur after 
$p_i+q_j$ in $\B{D_Z}$. 
Then $p_i+q_j$ is an element of $[p_{i}+q_{j-1},p_{i+1}+q_{z}]_D$.
Let 
\[
E_Z:=(p_i+q_j,\ p_u+q_{z}|\ u\in\mathbb{Z}/m\mathbb{Z}\ )
\] 
be a relative ideal of $S_Z$. Now  it follows that $p_i+q_j$ is in $[p_i+q_{z}, p_{i+1}+q_{z}]_E$, since the shift from $\varphi(p_{i+1}+q_{z})$ to $\varphi(p_{i+1}+q_{j-1})$ in $\mathbb{R}/\mathbb{Z}$ is the same as the shift from  $\varphi(p_{i}+q_{z})$ to $\varphi(p_{i}+q_{j-1})$; see Definition \ref{def:varphi}.

We may choose integers $x_0<x_1<x_2<x_0+b$, and $y_0>y_1\geq y_2>y_0-a$ such that
$p_i+q_z=\overline{(x_0,y_0)}$, $p_i+q_j=\overline{(x_1,y_1)}$ and $p_{i+1}+q_z=\overline{(x_2,y_2)}$.  
Let $m=\overline{(x_2,y_1)}$, $c=x_2-x_1$ and $d=y_1-y_2$.  Then we define $\delta(i,j)=a_i\otimes(b_j+ca)$.  
By our choice of $z$ we have $p_{i+1}+q_v$ is not in $[p_i+q_j,p_{i+1}+q_z]_E$ for all $v\neq z$; hence $m$ is not in $p_{i+1}+q_v+S_Z$ for $v\neq z$. We claim that $m$ is not in $p_u+q_z+S_Z$ for $u\neq i+1$ consider the following diagrams, which encompass the possible relationships between $(x_0,y_0)$, $(x_1,y_1)$ and $(x_2,y_2)$.  The dashed line corresponds to the preimage of $\B{E_Z}$ in $\mathbb{Z}^2$.

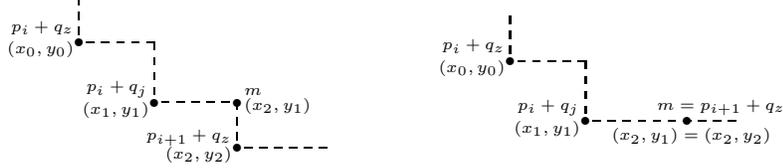
\setlength{\unitlength}{.5mm}
\begin{picture}(110,45) 
\put(25, 33){\circle*{2}}
\put(6,30){\tiny{$(x_0,y_0)$}}
\put(7,36){\tiny{$p_i+q_z$}}
\multiput(27,33)(4,0){5} 
{\line(1,0){2}}
\multiput(25,35)(0,4){3} 
{\line(0,1){2}}

\put(45, 17){\circle*{2}}
\put(26,14){\tiny{$(x_1,y_1)$}}
\put(27,20){\tiny{$p_i+q_j$}}
\multiput(47,17)(4,0){5} 
{\line(1,0){2}}
\multiput(45,19)(0,4){4} 
{\line(0,1){2}}

\put(67,5){\circle*{2}}
\put(48,2){\tiny{$(x_2,y_2)$}}
\put(43,7){\tiny{$p_{i+1}+q_z$}}
\multiput(67,7)(0,4){3} 
{\line(0,1){2}}
\multiput(69,5)(4,0){6} 
{\line(1,0){2}}

\put(67,17){\circle*{2}}
\put(69,19){\tiny{$m$}}
\put(69,15){\tiny{$(x_2,y_1)$}}

\end{picture}
\setlength{\unitlength}{.5mm}
\begin{picture}(110,45) 
\put(25, 28){\circle*{2}}
\put(6,25){\tiny{$(x_0,y_0)$}}
\put(7,31){\tiny{$p_i+q_z$}}
\multiput(27,28)(4,0){5} 
{\line(1,0){2}}
\multiput(25,30)(0,4){3} 
{\line(0,1){2}}

\put(45, 12){\circle*{2}}
\put(26,9){\tiny{$(x_1,y_1)$}}
\put(27,15){\tiny{$p_i+q_j$}}
\multiput(47,12)(4,0){10} 
{\line(1,0){2}}
\multiput(45,14)(0,4){4} 
{\line(0,1){2}}

\put(72,12){\circle*{2}}
\put(64,15){\tiny{$m=p_{i+1}+q_z$}}
\put(52,7){\tiny{$(x_2,y_1)=(x_2,y_2)$}}
\end{picture}

Since all points of the form $p_u+q_z$ with $u\neq i,i+1$ occur outside of the interval $[p_i+q_z,p_{i+1}+q_z]_E$,
it follows that $m$ is not in $p_u+q_z+S_Z$ for $u\neq i+1$. Therefore,
\[
\delta(i,j)=a_i\otimes(b_j+ca)\neq a_{i+1}\otimes(b_z+db)\in A\otimes_SB.
\]

{\bf Case 6:} Suppose $h=j$.  The construction for this case is similar to Case 5. We get that $p_{i-1}+q_{j}$ is the generator of  $A_Z+B_Z$ which precedes $p_i+q_j$ in $\G{A+B}$.

{\bf Case 6.1:} Suppose that $p_{i-1}+q_{j+1}$ is an element of $p_i+q_j+S_Z$. Then there exists 
$s$ in $S$ such that $a_i+b_j+s=a_{i-1}+b_{j+1}$.  Also, 
\[
\delta(i,j):=a_i\otimes(b_j+s)\neq a_{i-1}\otimes b_{j+1}\in A\otimes_SB.
\]

{\bf Case 6.2:} Suppose $p_{i-1}+q_{j+1}$ is not in $p_i+q_j+S_Z$. Given the relative ideal
\[
D_Z=(p_{i-1}+q_{j},\ p_i+q_j,\ p_{u}+q_{j+1}|\ u\in\mathbb{Z}/m\mathbb{Z}\ ),
\]
choose $z$ in $\mathbb{Z}/m\mathbb{Z}$ such that $p_z+q_{j+1}$ is the first element of the form 
$p_v+q_{j+1}$ with $v$ in $\mathbb{Z}/m\mathbb{Z}$ to occur after 
$p_i+q_j$ in $\B{D_Z}$. 
Choose integers $x_1,\ x_2,\ y_1$ and $y_2$ with $x_1<x_2<x_1+b$ and $y_1\geq y_2>y_1-a$, such that $p_i+q_j=\overline{(x_1,y_1)}$ and $p_z+q_{j+1}=\overline{(x_2,y_2)}$.  
Let $c=x_2-x_1$ and $d=y_1-y_2$. Then 
\[
\delta(i,j):=a_i\otimes(b_j+ca)\neq a_{z}\otimes(b_{j+1}+db)\in A\otimes_SB.
\]

{\bf Injectivity of $\delta$:} It remains to show that $\delta$ is injective.  In Case 1 we have $\delta(i,j)=a_i\otimes b_j\neq a_u\otimes b_v+s$ for all $(u,v)$ in $\mathbb{Z}/m\mathbb{Z}\times\mathbb{Z}/n\mathbb{Z}$ with $(u,v)\neq(i,j)$ and $s$ in $S$. 
Therefore, $\delta^{-1}(a_i\otimes b_j)=\{(i,j)\}$. In Cases $2, 4, 5$ and 
$6$ we have shown that $\delta(i,j)=a_i\otimes (b_j+s)$  such that there are integers $x_1<x_3\leq x_2$ and $y_1>y_2\geq y_3>y_1-a$ where $p_g+q_h=\overline{(x_1,y_1)}$, 
$\psi^{-1}(a_i+b_j+s)=\overline{(x_2,y_2)}$ and $p_i+q_j=\overline{(x_3,y_3)}$.   The following graph illustrates the relationship of these points.

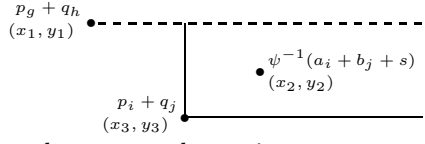
\setlength{\unitlength}{.5mm}
\begin{picture}(90,37) 
\put(25, 28){\circle*{2}}
\put(5,31){\tiny{$p_g+q_{h}$}}
\put(3,25){\tiny{$(x_1,y_1)$}}
\multiput(27,28)(4,0){22} 
{\line(1,0){2}}
\put(50,3){\circle*{2}}
\put(32,6){\tiny{$p_i+q_{j}$}}
\put(28,0){\tiny{$(x_3,y_3)$}}
\put(50,3){\line(1,0){63}}
\put(50,3) {\line(0,1){25}}

\put(70,15){\circle*{2}}
\put(72,17){\tiny{$\psi^{-1}(a_i+b_j+s)$}}
\put(72,11){\tiny{$(x_2,y_2)$}}
\end{picture}\\
Since the rectangular regions represented above containing are non-overlapping as $(i,j)$ varies, it follows that if $(i,j)$ and $(i',j')$ are distinct and each fit the criteria for either Case $2,\ 4,\ 5$ or $6$, then 
\[
\psi^{-1}\circ\chi\circ\delta(i,j)\neq \psi^{-1}\circ\chi\circ\delta(i',j');
\]
hence 
$\delta(i,j)\neq\delta(i',j')$. 
When $(i,j)$ is in Case 3 and $(i',j')$ is not in Case 1, if 
\[
\psi^{-1}\circ\chi\circ\delta(i,j)=\psi^{-1}\circ\chi\circ\delta(i',j'),
\]
then 
$(i',j')$ is in Case 4 and $(i',j')=(g,h)$, where $p_g+q_h$ precedes $p_i+q_j$ in $\B{A_Z+B_Z}$.  In this case we have shown that
\[
\delta(i,j)=a_i\otimes (b_j+db)\neq a_g\otimes (b_h+ca)=\delta(g,h).
\]
Therefore, $\delta$ is injective and the result follows.
\end{proof}

%
%
%
%
%

\section*{Acknowledgments}
During the course of this research the author was partially supported by a GAANN grant from the Department of Education and a GENIL-YTR grant from the University of Granada. I am grateful to Marco D'Anna, Pedro Garc\'ia-S\'anchez, Srikanth Iyengar, Manoj Kummini and Roger Wiegand 
for useful feedback about this research. Srikanth showed me how to simplify some of my early proofs using commutative diagrams. Roger convinced me to keep going when I would have otherwise given up, and Pedro helped me to see that my ideas involving semigroup ideals were significant independent of their relation to torsion. Additionally Roger and Pedro helped me to find financial support during the course of this research.

\end{document}